\documentclass[12pt, reqno]{amsart}
\usepackage{graphicx,colonequals} 
\usepackage{amsmath,amsfonts,amssymb,fullpage}
\usepackage[colorlinks=true,linkcolor=blue]{hyperref}

\newtheorem{theorem}{Theorem}[section]
\newtheorem{lemma}[theorem]{Lemma}
\newtheorem{prop}[theorem]{Proposition}

\newtheorem{remark}[theorem]{Remark}

\newcommand{\R}{\mathbb{R}}
\newcommand{\E}{\mathbb{E}}
\renewcommand{\P}{\mathbb{P}}
\renewcommand{\L}{\mathcal{L}}
\newcommand{\M}{\mathcal{M}}
\newcommand{\N}{\mathcal{N}}
\newcommand{\sgn}{\mathrm{sign}}

\renewcommand{\epsilon}{\varepsilon}

\title{A Lower Bound for Grothendieck's Constant}
\author{Steven Heilman}
\date{\today}

\begin{document}

\begin{abstract}
We show that Grothendieck's real constant $K_{G}$ satisfies $K_G\geq c+10^{-26}$, improving on the lower bound of $c=1.676956674215576\ldots$ of Davie and Reeds from 1984 and 1991, respectively.
\end{abstract}

\thanks{
Email: stevenmheilman@gmail.com\\
Supported by NSF Grant CCF 2448108\\
Department of Mathematics, University of Southern California, Los Angeles, CA 90089
}



\maketitle

\section{Introduction}

Grothendieck's real constant $K_{G}$ is the infimum over all $K\in(0,\infty)$ such that, for all positive integers $m,n$ and for every real $m\times n$ matrix $(a_{ij})$, we have
\begin{equation}\label{groeq}
\max_{\substack{x_{1},\ldots,x_{m}\\y_{1},\ldots,y_{n}}\in S^{m+n-1}}\sum_{i=1}^{m}\sum_{j=1}^{n}a_{ij}\langle x_{i},y_{j}\rangle
\leq K\cdot
\max_{\substack{\epsilon_{1},\ldots,\epsilon_{m}\\\delta_{1},\ldots,\delta_{n}}\in\{-1,1\}}\sum_{i=1}^{m}\sum_{j=1}^{n}a_{ij}\epsilon_{i}\delta_{j},
\end{equation}
where $\langle x,y\rangle=\sum_{i=1}^{d}x_{i}y_{i}$ for all $x,y\in\R^{d}$ and $S^{d-1}=\{x\in\R^{d}\colon\langle x,x\rangle=1\}$ for any $d\geq1$.

Inequality \eqref{groeq} was originally stated as an inequality of two different tensor norms \cite{groth53} (see also \cite[Section 3]{pisier12}), though the discretized formulation \eqref{groeq} was proven in \cite{linden77}.

Determining the exact value of $K_{G}$ remains a significant open problem since it was first posited in \cite{groth53}.  We can rephrase the problem of finding the constant $K_{G}$ as: what is the ``best'' way to ``round'' the vectors $x_{1},\ldots,x_{m},y_{1},\ldots,y_{n}$ to $\pm1$?  There are many good references on \eqref{groeq}, and its interest in combinatorics, functional analysis, Banach space theory, operator algebras and the Connes embedding problem, theoretical computer science, quantum mechanics, etc. such as \cite{pisier12,khot12}.  We briefly mention some interpretations of Grothendieck's constant in quantum information and theoretical computer science:
\begin{itemize}
\item Grothendieck's constant is the maximal quantum violation in Bell's inequality from quantum mechanics \cite{tsirelson87}.  That is, Grothendieck's inequality is a reformulated version of Bell's inequality.
\item Assuming the Unique Games Conjecture \cite{khot02}, it is NP-hard to approximate the right side of \eqref{groeq} within any constant smaller than Grothendieck's constant $K_{G}$ \cite{ragh09}.  The left side of \eqref{groeq} is a semidefinite program which can be computed efficiently, while the right side is an integer program.  So, \eqref{groeq} itself efficiently approximates the right side of \eqref{groeq} using its left side, within a constant factor $K_{G}$ \cite{alon04}.  The cut norm and the MAX-CUT problem are special cases of the right side of \eqref{groeq} \cite{alon04}.
\end{itemize}

\subsection{Upper Bounds on \texorpdfstring{$K_{G}$}{KG}}
\cite{groth53} originally proved that $K_{G}\leq\sinh(\pi/2)\approx2.3013$.  In \cite{krivine77}, Krivine showed that $K_{G}\leq\frac{\pi}{2\log(1+\sqrt{2})}\approx 1.78221397819$, and it was generally believed that this inequality should be an equality \cite{konig01}.  However, it was then shown in \cite{braverman13} that there is a $c'>0$ such that $$K_{G}<\frac{\pi}{2\log(1+\sqrt{2})} - c'.$$
An effective $c'$ was proven but not specified in \cite{braverman13}; an inspection of the argument seems to give $c'=10^{-500}$.  Krivine's argument \cite{krivine77} shows that, after ``preprocessing'' the vectors $x_{1},\ldots,x_{m},y_{1},\ldots,y_{n}$ by nonlinearly mapping them to a different Hilbert space (Fock space), we can then map those vectors to $\pm1$ by projecting them onto a Gaussian random vector, and taking the sign of this projected value.  (The nonlinear preprocessing removes the nonlinearity that appears after projecting onto the Gaussian.)  The argument of \cite{braverman13} instead projects the preprocessed vectors onto a random plane through the origin, and then (with probability $0<p<1$) applies a perturbation of the sign function on this two-dimensional plane to ``round'' the vectors to $\pm1$ (and with probability $1-p$ applies the sign function).  Also, rounding schemes of this form can obtain arbitrarily good approximations of $K_{G}$ \cite{regev14}.  In other words, finding the exact value of $K_{G}$ reduces to finding the best ``rounding scheme'' for vectors $x_{1},\ldots,x_{m},y_{1},\ldots,y_{n}$ in \eqref{groeq}.

In contrast to these upper bounds, lower bounds for $K_{G}$ result from finding a matrix where the ratio of both sides of \eqref{groeq} is far from $1$.

\subsection{Lower Bounds on \texorpdfstring{$K_{G}$}{KG}}
It was shown in \cite{groth53} that, if we add the restriction in \eqref{groeq} that $(a_{ij})$ is symmetric positive semidefinite, then the best constant in \eqref{groeq} becomes equal to $\pi/2\approx 1.57079632679$.  Consequently, $K_{G}\geq\pi/2$.  Achieving better lower bounds on $K_{G}$ therefore requires considering matrices that are not symmetric positive semidefinite (such as the infinite-dimensional matrix $R_{\lambda}$ in \eqref{rdef}).  Independently of each other, \cite{reeds93} and \cite{davie84} showed (using \eqref{rdef}) that
\begin{equation}\label{lbeq}
K_{G}\geq
1.676956674215576237077855078853\ldots
\end{equation}
While the upper bound of $\pi/(2\log(1+\sqrt{2}))$ was previously believed to be an equality \cite{krivine77}, the lower bound \eqref{lbeq} was not believed to be an equality.  Nevertheless, no better lower bound has been described in the literature since 1984.

\subsection{Brief Review of the Existing Lower Bound}\label{briefsec}
To prove \eqref{lbeq}, \cite{reeds93} and \cite{davie84} consider the following linear operator for some fixed $0<\lambda<1$.
\begin{equation}\label{rdef}
R_{\lambda}\colonequals P_{1}-\lambda I
\end{equation}
on the Hilbert space $L_{2}(\gamma_{n})=\{f\colon\R^{n}\to\R\colon \int_{\R^{n}}|f(x)|^{2}\gamma_{n}(x)dx<\infty\}$, where $I$ is the identity map (note that $R_{\lambda}$ is not positive definite),
$$
\gamma_{n}(x)\colonequals(2\pi)^{-n/2}e^{-\frac{1}{2}\|x\|_{\ell_{2}(\R^{n})}^{2}},\qquad\forall\,x\in\R^{n}.
$$
and $P_{1}$ is the projection onto the level one Hermite-Fourier coefficients, i.e.
\begin{equation}\label{p1def}
P_{1}f(x)\colonequals\int_{\R^{n}}\langle x,y\rangle f(y)\gamma_{n}(y)dy,\qquad\forall\,x\in\R^{n},\,\forall\,f\in L_{2}(\gamma_{n}).
\end{equation}
Then, interpreting $R_{\lambda}$ as an infinite-dimensional matrix in \eqref{groeq}, Davie and Reeds obtain
\begin{equation}\label{klb}
K_{G}\geq\sup_{n\geq1} \sup_{\lambda>0}
\,\frac{\sup_{g\colon\R^{n}\to B_{n}}\int_{\R^{n}}\|R_{\lambda}g(x)\|_{\ell_{2}(\R^{n})}\gamma_{n}(x)dx}{\sup_{f\colon\R^{n}\to[-1,1]}\int_{\R^{n}}|R_{\lambda}f(x)|\gamma_{n}(x)dx},
\end{equation}
where $\| x\|_{\ell_{2}(\R^{n})}\colonequals\langle x,x\rangle^{1/2}$ for all $x\in\R^{n}$ and $B_{n}\colonequals\{x\in\R^{n}\colon\|x\|_{\ell_{2}(\R^{n})}\leq1\}$.  (The equivalence of \eqref{groeq} to a ratio of operator norms of maps from $L_{\infty}$ to $L_{1}$ as in \eqref{klb} is shown e.g. in the introduction of \cite{reeds93} or \cite[Theorem 2.5]{pisier12})

They then show that, as $n\to\infty$, the numerator of \eqref{klb} converges to $1-\lambda$ (which can be seen by considering the function $g(x)=x/\|x\|$), and the denominator is $(\lambda/\eta)^{2}+\lambda(1-4\Phi(-\eta))$, where $\eta\in(0,1)$ satisfies $\sqrt{\frac{2}{\pi}}\eta e^{-\eta^{2}/2}=\lambda$.  Choosing the optimal $\lambda$ which is 
$$\lambda_*\approx0.1974790909949819604066867498464070553745\ldots$$ then yields
$$
K_{G}\stackrel{\eqref{klb}}{\geq}
\frac{1-\lambda_*}{(\lambda_*/\eta_*)^{2}+\lambda_*(1-4\Phi(-\eta_*))}\approx 1.676956674215576237077855078853\ldots,
$$
where $\Phi(t)\colonequals\int_{-\infty}^{t}e^{-z^{2}/2}dz/\sqrt{2\pi}$, $\forall$ $t\in\R$ and $\eta_*\in(0,1)$ satisfies $\sqrt{\frac{2}{\pi}}\eta_* e^{-\eta_*^{2}/2}=\lambda_*$.  (We mention in passing that choosing $\lambda=0$ recovers the weaker bound $K_{G}\geq\pi/2$.)
%

It is natural to try to consider further perturbations of $R_{\lambda}$ by e.g. adding some multiples of $P_{j}$, where $P_{j}$ is the projection onto level $j$ Hermite-Fourier coefficients, and then use \eqref{klb} for such a perturbation of $R_{\lambda}$.  The main difficulty then becomes computing the denominator of \eqref{klb} for the perturbation of $R_{\lambda}$.  That is, we want to find some perturbation of $R_{\lambda}$ such that the numerator of \eqref{klb} is the same ($1-\lambda$ as $n\to\infty$), but the denominator is slightly smaller.

Concerning this strategy, Reeds \cite{reeds93} comments:

``It would be interesting to attempt the direct computation of the norm of some
more general operator $\sum_{j}\alpha_{j}P_{j}$, where the $\alpha_{j}$ are not all of the same sign, but the methods of this paper probably do not extend beyond the case where only one
of the $\alpha_{j}$ is positive.''

\subsection{Our Contribution}

We find that, in fact, the operator of Reeds can be perturbed slightly into the form
\begin{equation}\label{rlbdef}
R_{\lambda,\beta} = P_{1}-\lambda I - \beta P_{3}.
\end{equation}
so that \eqref{klb} implies an improved lower bound on $K_{G}$.

\begin{theorem}[\textbf{Main}]\label{mainthm}
$$K_{G}\geq 10^{-26}+c.$$
\end{theorem}
Here $c=1.676956674215576237077855078853\ldots$ is the lower bound proven by \cite{davie84,reeds93}.

Instead of exactly computing the norm of $R_{\lambda,\beta}$ (which could be difficult since functions achieving the operator norm should have a ``high-dimensional'' structure, unlike the case $\beta=0$ where the optimizing functions have a ``one-dimensional'' structure), we instead indirectly estimate the operator norm of $R_{\lambda,\beta}$.

The argument proceeds by a simple perturbation, together with a characterization of the maximizers of $R_{\lambda}$, which is already apparent from \cite{reeds93}.  Let $\M$ be the set of measurable functions $f\colon\R^{n}\to[-1,1]$ that are maximizers of $\|R_{\lambda}f\|_{1}$.  (From Lemma \ref{lem:maximizers-are-pm1}, all such maximizers take values in $\{-1,1\}$ almost surely.)  Theorem \ref{mainthm} follows from the following:

\begin{itemize}
\item[(a)] There is some $\epsilon>0$ such that, if $f$ is within $L_{1}(\gamma_{n})$ distance $\epsilon$ from $\M$, then for all $0<\beta<10^{-10}$, $\|R_{\lambda,\beta}f\|_{1}\leq\|R_{\lambda}\|_{\infty\to1} - (.0057)\beta$.  (And $\epsilon=10^{-7}$ suffices.)
\item[(b)] If $f$ has $L_{1}(\gamma_{n})$ distance at least $\epsilon$ from $\M$, then $\|R_{\lambda,\beta}f\|_{1}\leq\|R_{\lambda}\|_{\infty\to1} -\epsilon^{2}10^{-12} + \beta$.

\end{itemize}
Here $\|R_{\lambda}\|_{\infty\to1}\colonequals\sup_{g\colon\R^{n}\to[-1,1]}\int_{\R^{n}}|R_{\lambda}g(x)|\gamma_{n}(x)dx$.  It is crucial in (a) that $\epsilon>0$ is fixed, i.e. that there is a fixed $\epsilon$ that works for all small $\beta>0$.

Combining (a) and (b) means, for $\beta$ small enough, 
$$\|R_{\lambda,\beta}\|_{\infty\to1}\leq\max(\|R_{\lambda}\|_{\infty\to1} - (.0057)\beta,\,\,\|R_{\lambda}\|_{\infty\to1} - \epsilon^{2}10^{-12}+\beta)
<\|R_{\lambda}\|_{\infty\to1}.$$
The strict inequality follows by choosing $\beta<10^{-12}\epsilon^{2}$.  Theorem \ref{mainthm} follows.  (We have presented Case (b) here informally for illustrative purposes; for the actual statement that we prove, we split into a few different cases, such as \eqref{seven1} and \eqref{beq}, culminating in \eqref{finalineq}.)

Most of our arguments use elementary inequalities.  Our main innovation is our strategy which reduces much of the problem to one-dimensional Gaussian analytic inequalities, by not directly considering maximizers of $\|R_{\lambda,\beta}\|_{\infty\to1}$.  Since our goal is to provide an explicit constant in Theorem \ref{mainthm}, this unfortunately increases the length of the paper.

\subsection{Technical Challenges}

One difficulty of part (a) of the strategy for proving Theorem \ref{mainthm} is that the set of maximizers $\M$ is infinite.  The set $\M$ can be characterized using the argument of Reeds \cite{reeds93} (see Lemma \ref{maxchar}), but we require our operator norm estimate for $R_{\lambda,\beta}$ to hold uniformly over all elements of (a neighborhood of) the infinite set $\M$.

The main difficulty with part (b) of our strategy for Theorem \ref{mainthm} is that modulus of continuity estimates for $R_{\lambda,\beta}$ are most naturally stated in terms of the $L_{1}$ norm of functions $f$ on $\R^{n}$, whereas modulus of continuity estimates for $R_{\lambda}$ are most naturally stated in terms of the $L_{1}$ norm of the (one-dimensional) conditional expectation $\E (f|P_{1}f)$.  The first notion of closeness implies the second, but the second might not imply the first.  In order to connect these two different estimates, we then need to show: if $\E (f|P_{1}f)$ is close to $\{\E(g|P_{1}g)\colon g\in\M\}$ in $L_{1}$ norm (on $\R$), then $f$ is close to $\M$ in $L_{1}$ norm (on $\R^{n}$).  However, the most natural way to prove this statement fails since these two norms are incomparable.  If $g$ is the closest element to $f$ in $\M$, then $\E (g|P_{1}g)$ and $\E(f|P_{1}f)$ could a priori be close or far in $L_{1}$ norm, since the moment vectors $\int_{\R^{n}}xf(x)\gamma_{n}(x)dx$ and $\int_{\R^{n}}xg(x)\gamma_{n}(x)dx$ might not be parallel.  Since we want to maintain explicit constants, one cannot simply rotate $g$ so that its moment vector is parallel to $f$ and hope that the rotation is still close to $f$.  To circumvent this issue, we instead explicitly construct a function $h$ that is close to $f$ in $L_{1}$ norm, with $h$ satisfying the conditions of $\M$ except it takes values in $[-1,1]$, and from that we infer there is some $g\in\M$ closer to $f$ in $L_{1}$ norm (so that $|g|=1$).  This strategy avoids reasoning about the closest $g\in\M$.

\subsection{Organization}

We prove (a) above in several steps of increasing generality.  After reviewing the argument of Reeds in Section \ref{sectwo}, we then prove the $\epsilon=0$, $n=1$ case of (a).  Then (a) is proven in \eqref{upperconstants}, using a uniform third moment bound from Section \ref{secfive} together with some stability estimates in Section \ref{secsix}.  That is, a more formal statement of (a) is \eqref{upperconstants}.

We then move on to proving (b) in Section \ref{seceight}.  This is proven by some one-dimensional triangle inequalities and rearrangement arguments in Sections \ref{seceight}, \ref{secnine} and \ref{sectail}, together with Lemma \ref{convlem}.  Then (b) is proven in Case 2 of the proof of Theorem \ref{mainthm} in Section \ref{seceleven}.

Section \ref{seceleven} proves Theorem \ref{mainthm} by combining (a) and (b).

\subsection{Conclusions and Future Directions}

The main contribution of this paper is to demonstrate that the previous best lower bounds on Grothendieck's constant from 1984 can be improved, by modifying a method of proof that was viewed skeptically in \cite{reeds93}.  Instead of explicitly computing the operator norm of $R_{\lambda,\beta}$, we instead approximate it by proving one estimate for a neighborhood of the maximizers of $R_{\lambda}$, and another estimate for functions outside this neighborhood.  Computing the operator norm of $R_{\lambda,\beta}$ directly seems more difficult (as \cite{reeds93} expresses), since these optimizers should be ``high-dimensional,'' whereas optimizers of $R_{\lambda}$ are ``one-dimensional.''  However, it would be interesting to more accurately compute the $\infty\to1$ norm of $R_{\lambda,\beta}$, since we only approximate it using upper bounds that are probably far from the truth.

We have made some effort to sharpen the constants in our proof, but presumably they can be improved further, leading to better lower bounds on $K_{G}$.  Moreover, a perturbation of the form
$$P_{1}-\lambda I - \beta P_{3} - \delta P_{5}$$
could lead to an even better estimate on $K_{G}$.  Obtaining the best possible bound of this type is left to future work.  Instead, our contribution is conceptual, as we identify a strategy to give better lower bounds on $K_{G}$ that was doubted to work properly in \cite{reeds93}.

The upper bounds on $K_{G}$ from \cite{braverman13} were proven with a roughly similar strategy, since they considered a perturbation of the sign function by a fifth degree Hermite polynomial.  They could have used a third order Hermite polynomial instead, but chose not to, perhaps due to additional complications added to the argument.  The details of the strategy of \cite{braverman13} are fairly different and perhaps more sophisticated than ours.  Also, \cite{braverman13} does not provide an explicit upper bound on $K_{G}$.  It seems their argument shows that $K_{G}<\frac{\pi}{2\log(1+\sqrt{2})}-10^{-500}$, though a sharper analysis might prove a better upper bound.  So, in terms of absolute constant improvements, our improvement on the lower bound seems larger than their improvement on the upper bound.

\subsection{Remark on the Complex Case}

\cite{davie84} also  obtains the best known lower bound on the complex Grothendieck constant using the same approach as sketched above, i.e. optimizing the $\infty\to1$ norm over $\lambda$ for the operator $P_{1}-\lambda I$.  It seems plausible our result could also give an improved lower bound for this constant.  That is, for appropriate $\lambda,\beta$, an improved lower bound on Grothendieck's complex constant should be achievable by analyzing the operator norm of $P_{1}-\lambda I - \beta P_{3}$.  Some preliminary computations suggest our approach could work in the complex setting.  However, we leave this problem to future work.

\subsection{Summary of Notation}

\begin{itemize}
\item $\gamma_{n}(x)\colonequals e^{-\frac{1}{2}\|x\|_{\ell_{2}(\R^{n})}^{2}}(2\pi)^{-n/2}$, $\forall$ $x\in\R^{n}$.
\item $\gamma_{n}(A)\colonequals\int_{A}\gamma_{n}(x)dx$ for all measurable $A\subset\R^{n}$.
\item $P_{k}$ denotes projection onto the level $k$ Hermite-Fourier coefficients in $L_{2}(\R^{n},\gamma_{n})$.
\item $R_{\lambda}\colonequals P_{1}-\lambda I$, $\lambda\in\R$.
\item $R_{\lambda,\beta}\colonequals P_{1}-\lambda I - \beta P_{3}$, $\lambda,\beta\in\R$.
\item $H_{3}(z)=z^{3}-3z$, $\forall$ $z\in\R$.
\item $\theta(z)=\theta_{g}(z)\colonequals \E[g| P_{1}g=\alpha z]$, $\forall$ $z\in\R$, $g\colon\R^{n}\to[-1,1]$.
\item $\alpha \colonequals \|\int_{\R^{n}}xg(x)\gamma_{n}(x)dx\|_{\ell_{2}(\R^{n})}
=|\int_{\R}z\theta(z)\gamma_{1}(z)dz|$.
\item $\eta\colonequals\lambda/\alpha$ (if $\alpha\neq0$).
\item $\|f\|_{p}\colonequals(\int_{\R^{n}}|f(x)|^{p}\gamma_{n}(x)dx)^{1/p}$, $\forall$ $p\geq1$, $\forall$ $f\colon\R^{n}\to[-1,1]$.
\item $L_{2}(\gamma_{n})\colonequals\{f\colon\R^{n}\to\R\quad\colon \|f\|_{2}<\infty\}.$
\item $\langle f,g\rangle\colonequals\int_{\R^{n}}f(x)g(x)\gamma_{n}(x)dx$, $\forall$ $f,g\in L_{2}(\gamma_{n})$.
\item $\| R_{\lambda}\|_{\infty\to1}=\sup\{\|R_{\lambda}f\|_{1}\colon\quad f\colon\R^{n}\to[-1,1]\}$.
\item $\M=\M_{\lambda,n}\colonequals\{g\colon\R^n\to[-1,1]\ \colon\ \|R_\lambda g\|_1=\|R_\lambda\|_{\infty\to 1}\}.$
\item $\Theta=\{\theta_{g}\colon g\in\M\}$.
\item $\lambda_*\approx 0.197479091$.
\item $\alpha_*\approx 0.7722165032$.
\item $\eta_*\colonequals\lambda_*/\alpha_*\approx 0.255730213$.
\end{itemize}

Unless otherwise stated, we will always use $\lambda\colonequals\lambda_*$.

\section{Review of Reeds}\label{sectwo}

We provide more details for the sketch in Section \ref{briefsec} of the argument of Reeds proving \eqref{lbeq}, since we will extend this argument.

Let $g\colon\R^{n}\to\{-1,1\}$ be measurable.  Denote $\alpha\colonequals\|P_{1}g\|_{2}$.  
Denote $\theta_{g}(z)\colonequals\E[g\,|\, P_{1}g=\alpha z]$, for all $z\in\R$.  ($\E$ is taken with respect to $\gamma_{n}$.)  Then by its definition
$$
\int_{\R}z\theta_{g}(z)\gamma_{1}(z)dz = \alpha.
$$
\begin{equation}\label{noreq}
\begin{aligned}
&\int_{\R^{n}}|R_{\lambda}g(x)|\gamma_{n}(x)dx
\stackrel{\eqref{rdef}}{=}\int_{\R}\gamma_{1}(z)\Big(\P(g=1|P_{1}g=\alpha z)|\alpha z - \lambda|
+\P(g=-1|P_{1}g=\alpha z)|\alpha z + \lambda|\Big) dz\\
&\qquad\qquad=\int_{\R}\gamma_{1}(z)\Big(\Big(\frac{1}{2}+\frac{1}{2}\theta_{g}(z)\Big)|\alpha z - \lambda|
+\Big(\frac{1}{2}-\frac{1}{2}\theta_{g}(z)\Big)|\alpha z + \lambda|\Big) dz\\
&\qquad\qquad=\int_{\R}\gamma_{1}(z)\Big(\frac{1}{2}\Big(|\alpha z - \lambda|+|\alpha z + \lambda|\Big)
+\frac{1}{2}\theta_{g}(z)\Big(|\alpha z - \lambda|-|\alpha z + \lambda|\Big) dz\\
&\qquad\qquad=2\lambda\gamma_{1}([0,\lambda/\alpha])
+2\alpha e^{-\frac{\lambda^{2}}{2\alpha^{2}}}/\sqrt{2\pi}
+\int_{\R}\gamma_{1}(z)\frac{1}{2}\theta_{g}(z)\Big(|\alpha z - \lambda|-|\alpha z + \lambda|\Big) dz.
\end{aligned}
\end{equation}
If $\alpha,\lambda$ are fixed, then $\theta_{g}$ only appears in the last term.  We therefore focus on optimizing that integral term (with $\alpha,\lambda$ fixed).
Denote
$$\psi(z)\colonequals\frac{1}{2}(|\alpha z - \lambda| - |\alpha z + \lambda|),\qquad\forall\,z\in\R.$$
If $z>0$, then $\psi(z)= -\alpha z1_{(z<\lambda/\alpha)}-\lambda 1_{(z\geq\lambda/\alpha)}$, and
\begin{equation}\label{psieq}
\psi(z)/z= -\alpha 1_{(|z|<\lambda/\alpha)}-\frac{\lambda}{|z|} 1_{(|z|\geq\lambda/\alpha)},\qquad\forall\,z\in\R.
\end{equation}

So, we consider the problem of maximizing
\begin{equation}\label{maxeq}
\int_{\R}\gamma_{1}(z)\psi(z)\theta(z)dz
\end{equation}
subject to the constraints that $\alpha>0$ is fixed, and
\begin{equation}\label{coneq}
\int_{\R}\theta(z)z\gamma_{1}(z)dz=\alpha,\qquad|\theta(z)|\leq1,\qquad\forall\,z\in\R,\qquad\,\alpha>0.
\end{equation}


Note that $\psi(z)/z$ inside the integrand of \eqref{maxeq} is the only different term when comparing that integrand and the integrand in \eqref{coneq}.  Since evidently $\psi(z)/z$ is even, strictly increasing when $z>\lambda/\alpha$, and constant when $0<z<\lambda/\alpha$ by \eqref{psieq}, a $\theta$ maximizing this optimization problem must take the value $\mathrm{sign}(z)$ for any $|z|>\lambda/\alpha$, if the constraint \eqref{coneq} can be satisfied for such a $\theta$; if not, then $\theta$ must be $\mathrm{sign}(z)$ for some $h>\lambda/\alpha$, and $-1$ otherwise.

More specifically, when $\alpha>0$, let $h>0$ be the unique value satisfying
\begin{equation}\label{alphadef}
\alpha = \Big(\int_{h}^{\infty}-\int_{0}^{h}\Big)2z\gamma_{1}(z)dz
=\sqrt{\frac{2}{\pi}}(2e^{-h^{2}/2}-1).
\end{equation}

If $h>\lambda/\alpha$, then there is a unique $\theta$ maximizing \eqref{maxeq} subject to \eqref{coneq} (namely, $\theta$ is odd and $\theta(z)=\mathrm{sign}(z-h)$ for all $z>0$).  But if $h<\lambda/\alpha$, then the maximizer $\theta$ is not unique, since $\psi(z)/z$ is constant for all $|z|<\lambda/\alpha$, and we only have $\theta(z)=\mathrm{sign}(z)$ for all $|z|>\lambda/\alpha$.  Denote $\eta\colonequals\lambda/\alpha$.  For such a $\theta$ we have
\begin{equation}\label{alphadef2}
\alpha \stackrel{\eqref{coneq}}{=}
\int_{-\eta}^{\eta}\theta(z)z\gamma_{1}(z)dz + \int_{|z|>\eta}|z|\gamma_{1}(z)dz
=\int_{-\eta}^{\eta}\theta(z)z\gamma_{1}(z)dz + 2\gamma_{1}(\eta).
\end{equation}
If we choose the right parameters, we can then characterize the maximizing functions $\theta$.
%

%
%

\begin{lemma}[Characterization of Maximizing Profiles]\label{maxchar}
Let $0<\lambda_{*}<1$.  Assume $h<\lambda_{*}/\alpha_{*}$.  Assume $\sqrt{\frac{2}{\pi}}\eta_{*} e^{-\eta_{*}^{2}/2}=\lambda_{*}$.  (So that $\gamma_{1}(\eta_{*})=\alpha_{*}/2$.)  Let $\Theta$ be the set of $\theta$ maximizing \eqref{maxeq} subject to constraints \eqref{coneq}.  Then $\Theta$ consists of all $\theta\colon\R\to[-1,1]$ such that
\begin{itemize}
\item $\theta(z) = \mathrm{sign}(z)$, $\forall$ $|z|>\eta_{*}$.
\item $\int_{-\eta_{*}}^{\eta_{*}}\theta(z)z\gamma_{1}(z)dz=0$.
\end{itemize}
\end{lemma}



Plugging in $\theta= - 1_{[0,h]}+1_{(h,\infty]}+1_{[-h,0)}-1_{(-\infty,-h)}$ into \eqref{noreq} and using \eqref{alphadef}, we get
$$
\int_{\R^{n}}|R_{\lambda}g(x)|\gamma_{n}(x)dx
=4\lambda\gamma_{1}([0,\lambda/\alpha])
-\alpha^{2}
+4\alpha e^{-\frac{\lambda^{2}}{2\alpha^{2}}}/\sqrt{2\pi}
-\lambda.
$$

The bound \eqref{lbeq} comes from optimizing this quantity over all $\alpha>0$, then plugging the result into \eqref{klb} and optimizing that over all $0<\lambda<1$.  Related to that calculation, we have
%
%



%

\begin{lemma}\label{taylorlem}
For any $0<\alpha<1$, let $F(\alpha)\colonequals 4\lambda\gamma_{1}([0,\lambda/\alpha]) - \alpha^{2}+4\alpha\gamma_{1}(\lambda/\alpha)-\lambda$.  Then
$$F'(\alpha) = 4\gamma_{1}(\lambda/\alpha) - 2\alpha.$$
$$F''(\alpha) = 4\lambda^{2}\alpha^{-3}\gamma_{1}(\lambda/\alpha)-2.$$
So, if $\alpha_{*}\approx .77$ satisfies $F'(\alpha_{*})=0$, and if $\lambda_*\approx.197479091$, then for all $\alpha$ with $|\alpha-\alpha_{*}|<.1$, 
$$F(\alpha)\leq F(\alpha_{*}) - \frac{9}{10}(\alpha-\alpha_{*})^{2}.$$
(And if $|\alpha-\alpha_*|>.1$, then $F(\alpha)\leq F(\alpha_*) - (9/10)(1/10)^2$.)
\end{lemma}



  



\section{Third Moment Bounds}

We begin proving our third moment bound for elements of $\M$ by treating the one-dimensional $n=1$ case.

Let $H_{3}(z)\colonequals z^{3}-3z$ for all $z\in\R$, so that $\int_{\R}(H_{3}(z))^{2}\gamma_{1}(z)\,dz = 6$.

Fix $\lambda>0$ and $\alpha>0$ and set
$
\eta\colonequals\lambda/\alpha.
$

When setting $F'(\alpha)=0$ in Lemma \ref{taylorlem}, we get, if $0<\eta_*<1$, then
\begin{equation}\label{eq:alpha=2phi}
\alpha_{*}=2\gamma_{1}(\eta_{*})
\quad\Longleftrightarrow\quad
\sqrt{\frac{2}{\pi}}\,\eta_{*} e^{-\eta_{*}^{2}/2}=\lambda_{*},
\end{equation}
and we will often assume \eqref{eq:alpha=2phi} below.  Let $\theta\colon\mathbb{R}\to[-1,1]$ satisfy
\begin{equation}\label{eq:theta-outside}
\theta(z)=\mathrm{sign}(z)\qquad\forall\,|z|\ge \eta_*.
\end{equation}
Under \eqref{eq:alpha=2phi}, the moment constraint $\int_{\R}z\theta(z)\gamma_{1}(z)dz=\alpha_*$
reduces to the single condition
\begin{equation}\label{eq:moment-zero}
\int_{-\eta_*}^{\eta_*} z\,\theta(z)\,\gamma_{1}(z)\,dz = 0.
\end{equation}


Lemma \ref{maxchar} shows that \eqref{eq:theta-outside} and \eqref{eq:moment-zero} characterize $\Theta$.

In this section we prove a uniform third Gaussian moment bound for any $\theta\colon\R\to\{-1,1\}$ satisfying \eqref{eq:theta-outside} and \eqref{eq:moment-zero}.

\begin{lemma}
\label{lem:tail-H3}
For every $\eta\ge 0$,
\begin{equation}\label{eq:B}
B\colonequals2\int_{\eta}^{\infty} H_{3}(z)\,\gamma_{1}(z)\,dz
=-2(1-\eta^{2})\,\gamma_{1}(\eta).
\end{equation}
\end{lemma}


\begin{lemma}
\label{lem:A-bound}
Assume $\theta$ satisfies $\int_{-\eta}^{\eta}\theta(z)z\gamma_{1}(z)dz=0$.
Define
\[
A(\theta)\colonequals\int_{-\eta}^{\eta} H_{3}(z)\,\theta(z)\,\gamma_{1}(z)\,dz.
\]
Then
\begin{equation}\label{eq:Amax}
|A(\theta)|\ \le\ \eta^{2}\big(\gamma_{1}(0)-\gamma_{1}(\eta)\big).
\end{equation}
\end{lemma}

\begin{proof}
By replacing $\theta$ with its odd part $(\theta -\theta(-\cdot))/2$, it suffices to prove the lemma for odd $\theta$.  Denote $\mu(z)\colonequals z\gamma_{1}(z)$ for all $z>0$.  For any $z\in(0,\eta)$, we have
\[
H_{3}(z)\gamma_{1}(z)=(z^{3}-3z)\gamma_{1}(z)=(z^{2}-3)\,z\gamma_{1}(z)=(z^{2}-3)\mu(z).
\]
Let $r(z)\colonequals z^{2}-3$. Then since $\theta$ and $H_{3}$ are odd, $H_{3}\theta$ is even, so
\[
\frac{A(\theta)}{2}=\int_{0}^{\eta} r(z)\,\theta(z)\mu(z)dz.
\]
Note that $r$ is increasing on $[0,\eta]$ and
\[
\max_{z\in[0,\eta]}r(z)-\min_{z\in[0,\eta]}r(z)
=\big(-3+\eta^{2}\big)-(-3)=\eta^{2}.
\]
Also $\mu(0,\eta)=\int_{0}^{\eta}z\gamma_{1}(z)\,dz=\gamma_{1}(0)-\gamma_{1}(\eta)$.  Since $|\theta(z)|\le 1$ $\forall$ $z\in\R$, define
$p(z)\colonequals(1+\theta)/2$ and $q(z)\colonequals(1-\theta)/2$, so $p(z),q(z)\in[0,1]$ and $p(z)-q(z)=\theta(z)$ for all $z\in\R$.  Therefore
\begin{equation}\label{difeq}
\int_{0}^{\eta} r(z)\theta(z)\mu(z)dz
=\int_{0}^{\eta} r(z)p(z)\mu(z)dz
-\int_{0}^{\eta} r(z)q(z)\mu(z)dz.
\end{equation}
The constraint $\int_{0}^{\eta}\theta(z)\mu(z)dz=0$ implies $\int_{0}^{\eta}p(z)\mu(z)dz=\int_{0}^{\eta}q(z)\mu(z)dz=\mu(0,\eta)/2$.
Then each term on the right side of \eqref{difeq} lies between $\min_{z\in[0,\eta]} r(z)\cdot\mu(0,\eta)/2$ and $\max_{z\in[0,\eta]} r(z)\cdot\mu(0,\eta)/2$,
so their difference is bounded by $\frac{1}{2}[\max_{z\in[0,\eta]} r(z)-\min_{z\in[0,\eta]} r(z)]\mu(0,\eta)=\frac{\eta^{2}}{2}(\gamma_{1}(0)-\gamma_{1}(\eta))$.
Multiplying by $2$ yields \eqref{eq:Amax}.
\end{proof}

\begin{lemma}\label{mombd}
Assume $\theta$ satisfies $\theta(z)=\mathrm{sign}(z)$ for all $|z|>\eta$ for some $\eta<1/2$.  Then
$$\int_{-\infty}^{\infty}\theta(z)z\gamma_{1}(z)dz\neq0.$$
Moreover, if $\bar{\theta}\colon\R\to[-1,1]$ satisfies $\|\theta-\bar{\theta}\|_{2}<1/100$, then 
$$
\Big|\int_{-\infty}^{\infty}\overline{\theta}(z)z\gamma_{1}(z)dz\Big|>.6.
$$
\end{lemma}
\begin{proof}
We have
$$
\int_{-\infty}^{\infty}\theta(z)z\gamma_{1}(z)dz
\geq2\int_{0}^{\eta}(-1)z\gamma_{1}(z)dz
+2\int_{\eta}^{\infty}z\gamma_{1}(z)dz
=2(2e^{-\eta^{2}/2}-1)/\sqrt{2\pi}\geq.6104\ldots>0,
$$
since $\eta<1/2$.  For the second conclusion, note that
$$\Big|\int_{\R}z(\theta(z)-\bar{\theta}(z))\gamma_{1}(z)dz\Big|
\leq\|\theta-\bar{\theta}\|_{2}\Big(\int_{\R}z^{2}\gamma_{1}(z)dz\Big)^{1/2}<1/100.$$
So $\int_{\R}z\bar{\theta}(z)\gamma_{1}(z)dz>.6$ by the reverse triangle inequality.
\end{proof}

\section{Maximizers have absolute value 1 a.e.}

We now build upon Lemma \ref{maxchar} and further characterize the maximizing set $\M$.  Let


$$
\M = \M_{\lambda,n}\colonequals\{g\colon\R^n\to[-1,1]:\ \|R_\lambda g\|_1=\|R_\lambda\|_{\infty\to1}\}.
$$

\begin{lemma}[Every maximizer of $\|R_\lambda\|_{\infty\to1}$ is $\{\pm1\}$-valued]\label{lem:maximizers-are-pm1}

For every $g\in\mathcal M$ we have $|g(x)|=1$ for $\gamma_n$-a.e.\ $x$ (hence $g\in\{\pm1\}$ a.e.).
\end{lemma}

\begin{proof}
Fix $g\in\mathcal M$. Let $a\colonequals R_\lambda g$.
Choose $h:\R^n\to\{\pm1\}$ with $h(x)\in\mathrm{sign}(a(x))$ for all $x$. 
Then $|h|\le 1$ and
\begin{equation}\label{eq:choose-h}
\|R_\lambda g\|_1=\int_{\R^{n}} |a(x)|\,\gamma_{n}(x)dx=\int_{\R^{n}} h(x)a(x)\gamma_{n}(x)dx=\langle h, R_\lambda g\rangle.
\end{equation}

Since $P_1$ is an orthogonal projection, it is self-adjoint on $L_2(\gamma_n)$; hence $R_\lambda$ is self-adjoint as well.
Therefore \eqref{eq:choose-h} becomes
\[
\|R_\lambda g\|_1=\langle R_\lambda h, g\rangle.
\]
Because $|g|\le 1$, we have the pointwise bound $(R_\lambda h)g\le |R_\lambda h|$, hence
\begin{equation}\label{eq:upper-by-Th}
\|R_\lambda g\|_1=\langle R_\lambda h, g\rangle \le \int_{\R^{n}} |R_\lambda h(x)|\,\gamma_{n}(x)dx=\|R_\lambda h\|_1.
\end{equation}

On the other hand, since $g$ is a maximizer,
\[
\|R_\lambda h\|_1 \le \|R_\lambda\|_{\infty\to1}=\|R_\lambda g\|_1.
\]
Combining with \eqref{eq:upper-by-Th} and \eqref{eq:choose-h} forces equality:
\[
\langle R_\lambda h, g\rangle=\int_{\R^{n}} |R_\lambda h(x)|\,\gamma_{n}(x)dx.
\]
But equality in $(R_\lambda h)g\le |R_\lambda h|$ implies that for $\gamma_n$-a.e.\ $x$ with $(R_\lambda h)(x)\neq 0$,
we must have $g(x)=\mathrm{sign}((R_\lambda h)(x))$, and in particular $|g(x)|=1$ there.

It therefore remains to show that $R_\lambda h\neq 0$ almost everywhere.
Write $P_1 h(x)=\langle m, x\rangle$ $\forall$ $x\in\R^{n}$ where $m\colonequals\E[h(X)X]\in\R^n$. Then
\[
R_\lambda h(x)\stackrel{\eqref{rdef}}{=}\langle m,x\rangle-\lambda h(x).
\]
On the set $\{h=+1\}$ we have $R_\lambda h=\langle m,x\rangle-\lambda$, whose zero set is contained in the hyperplane
$\{\langle m,x\rangle=\lambda\}$; similarly on $\{h=-1\}$ the zero set is contained in $\{\langle m,x\rangle=-\lambda\}$.
Each hyperplane has $\gamma_n$-measure $0$, hence $\gamma_n(\{R_\lambda h=0\})=0$.
Therefore $|g|=1$ almost everywhere.  So, $R_{\lambda}h\neq0$ almost everywhere, completing the proof.
\end{proof}

\begin{lemma}\label{convlem}
Let $f\colon\R^{n}\to\{-1,1\}$ be measurable, and let
\begin{flalign*}
\mathcal{N}
&\colonequals\Big\{h\colon\R^{n}\to[-1,1]\colon\int_{\R^{n}}x_{i}h(x)\gamma_{n}(x)dx=0,\,\forall\,2\leq i\leq n,\,\wedge\,\int_{|x_{1}|<\eta_{*}}xh(x)\gamma_{n}(x)dx=0\\
&\hspace{4cm}\wedge h(x)=\mathrm{sign}(x_{1}),\text{ for a.e. }x\in\R^{n}\text{ with }\,|x_{1}|>\eta_*\Big\}.
\end{flalign*}
Let $h\in\N$.  Then there exists $g\in\M$ such that
$$\|f-g\|_{1}\leq\|f-h\|_{1}$$
\end{lemma}
\begin{proof}
Since $|f|=1$ and $|h|\leq1$, we have the pointwise equality $|f-h|=1-fh$, so that
\begin{equation}\label{fhdif}
\|f-h\|_{1}=1-\langle f,h\rangle.
\end{equation}
Consider then the linear functional $\L\colon\N\to\R$ defined by
$$\L(h_{0})\colonequals\langle f,h_{0}\rangle,\qquad\forall\,h_{0}\in\N.$$

Note that $\mathcal{N}$ is a convex and compact subset with respect to the weak topology on the Hilbert space $L_{2}(\gamma_{n})$, since it is norm bounded and weakly closed and the moment constraints such as $h_{0}\mapsto \int_{\R^{n}}x_{i}h_{0}(x)\gamma_{n}(x)dx$ are weakly continuous and linear.  Since $\L$ is then a continuous linear functional on the convex compact set $\N$, $\L$ achieves its maximum on $\N$ at an extreme point of $\N$ by the Bauer Maximum Principle.  We therefore characterize the extreme points of $\N$.

The extreme points of $\mathcal{N}$ are contained in the set of elements of $\N$ taking values in $\{-1,1\}$ almost surely.  To see this, note that if an extreme point $g\in\N$ satisfies $\gamma_{n}(|g|\neq1)>0$, then we can exhibit $g$ as an element of a line in $\mathcal{N}$ by adding a small multiple of a bounded function $\widetilde{g}$ supported in the set $A\colonequals\{|x_{1}|<\eta_*\}\cap\{|g|<1-\epsilon_{0}\}$ satisfying $\int_{|x_{1}|<\eta_*}x\widetilde{g}(x)\gamma_{n}(x)dx=0$ for some $0<\epsilon_{0}<1$ small enough such that $\gamma_{n}(|g|<1-\epsilon_{0})>0$.  (Such a $\widetilde{g}$ exists by e.g. letting $A_{1},\ldots,A_{n+1}\subset A$ be disjoint sets with positive measure, then choosing by linear algebra constants $a_{1},\ldots,a_{n+1}\in\R$ not all zero such that $\widetilde{g}=\sum_{i=1}^{n+1}a_{i}1_{A_{i}}$ and $\int_{\R^{n}}x\widetilde{g}(x)\gamma_{n}(x)dx=0$, which amounts to solving $n$ equations in $n+1$ unknowns $a_{1},\ldots,a_{n+1}$.)  Then $|g+t\widetilde{g}|\leq1$  and $g+t\widetilde{g}\in\N$ for all $|t|<\epsilon_0 / \|\widetilde{g}\|_{\infty}$.

That is, the extreme points of $\N$ are elements of $\M$, by Lemma \ref{maxchar}.  Therefore, $\L$ achieves its maximum on $\N$ on the subset $\M\cap\N$.  That is,
$$\sup_{h_{0}\in\N}\L(h_{0})=\max_{g\in\M\cap\N}\L(g).$$
From \eqref{fhdif}, we conclude that
$$\inf_{h_{0}\in\N}\|f-h_{0}\|_{1}=\min_{g\in\M\cap\N}\|f-g\|_{1}.$$
That is, there exists some $g\in\M$ such that
$$\|f-g\|_{1}\leq\|f-h\|_{1}.$$

\end{proof}

\section{Derivative Bound}\label{secfive}

The main result of this section is Lemma \ref{lemmaipbd}, which gives a lower bound on the $\langle P_{3}g,\mathrm{sign}R_{\lambda_*}g\rangle$  for all $g\in\M$.

\begin{lemma}[Lower bound for $\langle \sgn(R_\lambda g),\,P_3 g\rangle$ via $(p,s_1,t_2)$]\label{lem:P3pairing-polished}

Assume we are at the \emph{Reeds point}, i.e.\ there exist $\eta_*>0$ and $\alpha_*>0$ such that
\begin{equation}\label{eq:reeds-point}
\sqrt{\frac{2}{\pi}}\,\eta_* e^{-\eta_*^2/2}=\lambda_*<1,
\qquad
\alpha_*=\frac{\lambda_*}{\eta_*}=2\gamma_{1}(\eta_*).
\end{equation}
Let $g\colon\R^n\to\{\pm1\}$ satisfy $g\in\M$ (by Lemma \ref{lem:maximizers-are-pm1} there is no loss of generality.)

Let $X\sim\gamma_n$ ($X\in\R^{n}$ is a mean zero Gaussian with identity covariance matrix) and set
\[
m\colonequals\E[g(X)\,X]\in\R^n,\qquad u\colonequals\frac{m}{\|m\|_2}\in S^{n-1},\qquad Z\colonequals\langle u,X\rangle\sim N(0,1),
\]
($m\neq 0$ by Lemma \ref{mombd} since $\eta_*<1/2$) so that
\[
P_1 g(X)=\langle m,X\rangle=\alpha_* Z
\qquad\text{with}\qquad \alpha_*=\|P_1 g\|_2=\|m\|_2,
\]
Let $Y\colonequals X-Zu\in\R^{n-1}$ denote the orthogonal complement so that $(Z,Y)$ are independent Gaussians.  Let
\[
f\colonequals\sgn(R_{\lambda_*} g)=\sgn(\alpha_* Z-\lambda_* g)\in\{\pm1\}\quad\text{a.e.}
\]
Let $H_3(z)\colonequals z^3-3z$ $\forall$ $z\in\R$ (so $\E[H_3(Z)^2]=6$). Define
\begin{equation}\label{badefs}
B\colonequals2\int_{\eta_*}^{\infty} (z^3-3z)\gamma_{1}(z)\,dz=-2(1-\eta_*^2)\gamma_{1}(\eta_*),
\qquad
A_{\max}\colonequals\eta_*^2(\gamma_{1}(0)-\gamma_{1}(\eta_*)),
\end{equation}
\begin{equation}\label{kdef}
\kappa_Q\colonequals\frac{B^2-A_{\max}^2}{6}.
\end{equation}
Finally define the inner-region quantities (with $Z\sim N(0,1)$)
\begin{equation}\label{pstdefs}
p\colonequals\P(|Z|<\eta_*),\qquad s_1\colonequals\E\big[|Z|\,1_{\{|Z|<\eta_*\}}\big],\qquad
t_2\colonequals\E\big[Z^2\,1_{\{|Z|<\eta_*\}}\big].
\end{equation}

Then:
\begin{enumerate}
\item[(i)] The following explicit formulas hold:
\begin{equation}\label{eq:p-s1-t2}
p=2\Phi(\eta_*)-1,
\qquad
s_1=2(\gamma_{1}(0)-\gamma_{1}(\eta_*)),
\qquad
t_2=p-2\eta_*\gamma_{1}(\eta_*),
\end{equation}
where $\Phi(\eta_*)\colonequals\int_{-\infty}^{\eta_*}e^{-z^{2}/2}dz/\sqrt{2\pi}$.

\item[(ii)] One has the decomposition
\begin{equation}\label{eq:P3pairing-decomp}
\langle f,P_3 g\rangle
=
\underbrace{\frac{1}{6}\E[g(X)H_3(Z)]\,\E[f(X)H_3(Z)]}_{\ge\,\kappa_Q}
\ -\ \|P_\perp g\|_2^2,
\end{equation}
where $P_\perp$ denotes the orthogonal projection of the third chaos onto the direct sum of the
\emph{transverse} components
\[
\Big(\mathcal H_1(Z)\otimes \mathcal H_2(Y)\Big) 
\oplus
\Big(\mathcal H_2(Z)\otimes \mathcal H_1(Y)\Big) 
\oplus
\Big(\mathcal H_0(Z)\otimes \mathcal H_3(Y)\Big).
\]
(here $\mathcal{H}_{k}(Z)$ denotes the span of $H_{k}(Z)$) and accordingly write
\begin{equation}\label{pdefs}
P_3=P_{3,0}+P_{\perp}
=P_{3,0}+P_{1,2}+P_{2,1}+P_{0,3}
\end{equation}
for the orthogonal projections, where $P_{3,0}$ projects onto the span of $\mathcal{H}_{3}(Z)\otimes\mathcal{H}_{0}(Y)$.

\item[(iii)] The transverse mass satisfies the explicit bound
\begin{equation}\label{eq:Pperp-bound}
\|P_\perp g\|_2^2 \le p^2+s_1^2+\frac12\,t_2^2.
\end{equation}

\item[(iv)] Consequently,
\begin{equation}\label{eq:P3pairing-final}
\boxed{
\ \ \langle f,P_3 g\rangle \ \ge\ \kappa_Q-\Big(p^2+s_1^2+\tfrac12 t_2^2\Big).\ \ }
\end{equation}
\end{enumerate}

\end{lemma}

\begin{proof}
\textbf{Step 1: formulas for $(p,s_1,t_2)$.}
By symmetry,
\[
p\stackrel{\eqref{pstdefs}}{=}\P(-\eta_*<Z<\eta_*)=\Phi(\eta_*)-\Phi(-\eta_*)=2\Phi(\eta_*)-1.
\]
Also
\[
s_1 \stackrel{\eqref{pstdefs}}{=} 2\int_0^{\eta_*} z\gamma_{1}(z)\,dz = 2(\gamma_{1}(0)-\gamma_{1}(\eta_*)),
\]
Finally, we deduce \eqref{eq:p-s1-t2} from
\begin{flalign*}
t_2 
&\stackrel{\eqref{pstdefs}}{=} 2\int_0^{\eta_*} z^2\gamma_{1}(z)\,dz
=2\Big(-\eta_{*}\gamma_{1}(\eta_*)+\int_0^{\eta_*} \gamma_{1}(z)\,dz\Big)\\
&=2\Big(\Phi(\eta_*)-\tfrac12-\eta_*\gamma_{1}(\eta_*)\Big)
\stackrel{\eqref{pstdefs}}{=}p-2\eta_*\gamma_{1}(\eta_*).
\end{flalign*}

\smallskip
\textbf{Step 2: a pointwise identity on the ``flat'' region.}
Since $g$ maximizes $\|R_{\lambda_*}g\|_{1}$, Lemma \ref{maxchar} implies that the conditional bias
$\theta(z)\colonequals\E[g\mid Z=z]$ satisfies $\theta(z)=\sgn(z)$ for $|z|>\eta_*$; since $g\in\{\pm1\}$ this forces
$g=\sgn(Z)$ almost surely on $\{|Z|>\eta_*\}$. In particular, on $\{|Z|>\eta_*\}$ the function $g$ depends only on $Z$.

Since $g\in\{\pm1\}$ and $|Z|<\eta_*$ implies $|\alpha_* Z|<\lambda_*$ by \eqref{eq:reeds-point}, we have pointwise
\begin{equation}\label{fsign}
f=\sgn(\alpha_* Z-\lambda_* g)=
\begin{cases}
g & \text{, if }|Z|>\eta_*\\
-g & \text{, if }|Z|<\eta_*.
\end{cases}
\end{equation}

\smallskip
\textbf{Step 3: transverse pairing is a negative square.}
Let $\psi(Z,Y)$ be 
any Hermite polynomial of
total degree $3$ that is a nonconstant function of $Y$. Then \begin{equation}\label{psizero}
\E[\psi\mid Z]=0
\end{equation}
because its $Y$-part has mean $0$.
Using that $g$ depends only on $Z$ on $\{|Z|>\eta_*\}$, we obtain
\begin{equation}\label{gzero}
\E[g\psi\,1_{\{|Z|>\eta_*\}}]
=\E\big[\E[g\psi\,1_{\{|Z|>\eta_*\}}|Z]\big]
=\E\big[g\,1_{\{|Z|>\eta_*\}}\E[\psi\mid Z]\big]=0,
\end{equation}
hence $\langle g,\psi\rangle=\E[g\psi\,1_{\{|Z|<\eta_*\}}]$.  Similarly, $\E[f\psi\,1_{\{|Z|>\eta_*\}}]=0$ by \eqref{fsign}.  On $\{|Z|<\eta_*\}$ we have $f=-g$ by \eqref{fsign}, hence
\begin{equation}\label{fgneg}
\langle f,\psi\rangle=\E[f\psi\,1_{\{|Z|<\eta_*\}}]=-\E[g\psi\,1_{\{|Z|<\eta_*\}}]=-\langle g,\psi\rangle.
\end{equation}
Let $\{\psi_j\}$ be an orthonormal basis of the transverse third chaos. Then
\[
\langle f,P_\perp g\rangle
=\sum_j \langle f,\psi_j\rangle\langle g,\psi_j\rangle
\stackrel{\eqref{fgneg}}{=}\sum_j (-\langle g,\psi_j\rangle)\langle g,\psi_j\rangle
=-\sum_j \langle g,\psi_j\rangle^2
=-\|P_\perp g\|_2^2.
\]
This gives
\[
\langle f,P_3 g\rangle
\stackrel{\eqref{pdefs}}{=}\langle f,P_{3,0}g\rangle+\langle f,P_\perp g\rangle
=\langle f,P_{3,0}g\rangle-\|P_\perp g\|_2^2,
\]
where $P_{3,0}$ 
is 
defined after \eqref{pdefs}.

\smallskip
\textbf{Step 4: lower bound on the zonal component.}
Since the zonal third chaos is one-dimensional, $P_{3,0}g$ is a scalar multiple of $H_3(Z)$:
\[
P_{3,0}g
\stackrel{\eqref{pdefs}}{=}
\frac{\E[gH_3(Z)]}{6}\,H_3(Z).
\]
Therefore
\[
\langle f,P_{3,0}g\rangle=\frac{\E[gH_3(Z)]\,\E[fH_3(Z)]}{6}.
\]
Then \eqref{fsign} and Lemma \ref{maxchar} imply
$\E[gH_3(Z)]=B+A$ and $\E[fH_3(Z)]=B-A$, where $B$ is defined in \eqref{eq:B}
and $A$ is the inner-region term
\[
A\colonequals\int_{-\eta_*}^{\eta_*} (z^3-3z)\theta(z)\gamma_{1}(z)\,dz,
\qquad \theta(z)=\E[g\mid Z=z]\in[-1,1].
\]
Lemma \ref{maxchar} says
$\int_{-\eta_*}^{\eta_*} z\theta(z)\gamma_{1}(z)\,dz=0$ and Lemma~\ref{lem:A-bound} says
$|A|\stackrel{\eqref{eq:Amax}}{\leq}\eta_*^2(\gamma_{1}(0)-\gamma_{1}(\eta_*))$.  So
\[
\langle f,P_{3,0}g\rangle=\frac{(B+A)(B-A)}{6}=\frac{B^2-A^2}{6}\ge \frac{B^2-A_{\max}^2}{6}\stackrel{\eqref{kdef}}{=}\kappa_Q.
\]

\smallskip
\textbf{Step 5: bounding $\|P_\perp g\|_2^2$ in terms of $(p,s_1,t_2)$.}
Let $I\colonequals1_{\{|Z|<\eta_*\}}$. Define the $Y$-measurable functions
\begin{equation}\label{mdefs}
m_0(Y)\colonequals\E[gI\mid Y],\qquad
a_1(Y)\colonequals\E[gZI\mid Y],\qquad
a_2(Y)\colonequals\frac{1}{\sqrt2}\E[gZ^2 I\mid Y].
\end{equation}
Since $|g|\le 1$ and $Z\perp Y$, we have the pointwise bounds (via \eqref{pstdefs})
\begin{equation}\label{mbounds}
|m_0(Y)|\le \E[I]\stackrel{\eqref{pstdefs}}{=}p,\qquad
|a_1(Y)|\le \E[|Z|I]\stackrel{\eqref{pstdefs}}{=}s_1,\qquad
|a_2(Y)|\le \frac{1}{\sqrt2}\E[Z^2 I]\stackrel{\eqref{pstdefs}}{=}\frac{t_2}{\sqrt2}.
\end{equation}
Now consider each transverse block in the orthogonal decomposition of the third chaos:

\emph{(0,3) block.} For any $\psi\in \mathcal H_3(Y)$, 
\[
\langle g,\psi(Y)\rangle
\stackrel{\eqref{gzero}}{=}\E[gI\,\psi(Y)]
\stackrel{\eqref{mdefs}}{=}\E[m_0(Y)\psi(Y)]
=\langle m_{0}(Y),\psi(Y)\rangle.
\]
Thus $P_{0,3}g\stackrel{\eqref{pdefs}}{=}P_{\mathcal H_3(Y)}(m_0)$ and
\begin{equation}\label{p03bd}
\|P_{0,3}g\|_2\le \|m_0\|_2\le \|m_0\|_\infty
\stackrel{\eqref{mbounds}}{\le} p.
\end{equation}

\emph{(1,2) block.} For any $\psi\in\mathcal H_2(Y)$, we similarly have
\[
\langle g,Z\psi(Y)\rangle
\stackrel{\eqref{gzero}}{=}\E[gZI\,\psi(Y)]
\stackrel{\eqref{mdefs}}{=}\E[a_1(Y)\psi(Y)]
=\langle a_{1}(Y),\psi(Y)\rangle.
\]
Thus $P_{1,2}g\stackrel{\eqref{pdefs}}{=}Z\cdot P_{\mathcal H_2(Y)}(a_1)$ and since $\|Z\|_2=1$,

\begin{equation}\label{p12bd}
\|P_{1,2}g\|_2=\|P_{\mathcal H_2(Y)}(a_1)\|_2\le \|a_1\|_2\le \|a_1\|_\infty\stackrel{\eqref{mbounds}}{\le} s_1.
\end{equation}

\emph{(2,1) block.} Let $\ell(Y)\in\mathcal H_1(Y)$ be any linear form in the $Y$-coordinates.
Since $P_1 g$ is parallel to $u$ (the $Z$-direction), we have $\E[g\,\ell(Y)]=0$.
Also, $\E[gZ^{2}\ell(Y)1_{|Z|>\eta_*}]=0$ because $\ell(Y)$ has mean $0$ and $g$ is constant when $|Z|>\eta_*$.
So, writing $h_2(Z)\colonequals(Z^2-1)/\sqrt2$,
\[
\langle g, h_2(Z)\,\ell(Y)\rangle
=\frac{1}{\sqrt2}\E[g\cdot(Z^2-1)\ell(Y)]
=\frac{1}{\sqrt2}\E[gZ^2 I\,\ell(Y)]
\stackrel{\eqref{mdefs}}{=}\E[a_2(Y)\,\ell(Y)].
\]
Hence $P_{2,1}g=h_2(Z)\cdot P_{\mathcal H_1(Y)}(a_2)$ and since $\|h_2\|_2=1$,
\begin{equation}\label{p21bd}
\|P_{2,1}g\|_2=\|P_{\mathcal H_1(Y)}(a_2)\|_2\le \|a_2\|_2\le \|a_2\|_\infty
\stackrel{\eqref{mbounds}}{\le} \frac{t_2}{\sqrt2}.
\end{equation}

By orthogonality of the three transverse blocks,
\[
\|P_\perp g\|_2^2\stackrel{\eqref{pdefs}}{=}\|P_{0,3}g\|_2^2+\|P_{1,2}g\|_2^2+\|P_{2,1}g\|_2^2
\stackrel{\eqref{p03bd}\wedge\eqref{p12bd}\wedge\eqref{p21bd}}{\le} 
p^2+s_1^2+\frac12 t_2^2,
\]
which is \eqref{eq:Pperp-bound}. Combining with \eqref{eq:P3pairing-decomp} completes the proof.
\end{proof}

We can now plug constants into \eqref{eq:P3pairing-final} and verify the lower bound is positive.

\begin{lemma}[Formal Derivative]\label{lemmaipbd}
Let $g\colon\R^n\to\{\pm1\}$ satisfy $g\in\M$.  Then
\[
\langle \mathrm{sign}(R_{\lambda_*}g),P_3 g\rangle\ge 0.0868120048-0.0414080847\approx 0.0454039202.
\]
\end{lemma}

\begin{proof}
If $\lambda_*=0.197479091$ then $\eta_*\stackrel{\eqref{eq:reeds-point}}{\approx} 0.255730213173163$, $\alpha_*\stackrel{\eqref{eq:reeds-point}}{\approx}0.772216503281451$,
$B\stackrel{\eqref{badefs}}{\approx} -0.721715133242779$, $A_{\max}\stackrel{\eqref{badefs}}{\approx} 0.000839319067615$, so
$\kappa_Q\stackrel{\eqref{kdef}}{\approx} 0.086812004849191$.
Moreover $p\approx 0.201840836034193$, $s_1\approx 0.0256680575214142$, $t_2\approx 0.00436174503419317$,
by \eqref{eq:p-s1-t2}, so $p^2+s_1^2+\tfrac12 t_2^2\approx 0.0414080846777763$. Therefore \eqref{eq:P3pairing-final} concludes the proof.
\end{proof}


\begin{lemma}[Pointwise sign-flip inequality]\label{lem:signflip3}
For all $a,b\in\R$ and all $\beta\ge 0$,
\begin{equation}\label{eq:signflip}
|a-\beta b|
\ \le\
|a|\ -\ \beta\,\mathrm{sign}(a)\,b\ +\ 2\beta\,|b|\,\mathbf{1}_{\{|a|\le \beta |b|\}}.
\end{equation}
\end{lemma}

\begin{proof}
If $|a|>\beta|b|$, then $a$ and $a-\beta b$ have the same sign, hence
$|a-\beta b|=|a|-\beta\,\mathrm{sign}(a)\,b$ and the indicator term vanishes.
If $|a|\le \beta|b|$, then 
$|a-\beta b|\le |a|+\beta|b|=|a|-\beta\,\mathrm{sign}(a)\,b + \beta(|b|+\mathrm{sign}(a)b)
\le |a|-\beta\,\mathrm{sign}(a)\,b + 2\beta|b|$.
\end{proof}

\section{Fixed Neighborhood with Norm Drop}\label{secsix}

Informally, Lemma \ref{lemmaipbd} implies that, for small enough $\beta>0$, if $g$ is in a neighborhood of $\M$, then $\|R_{\lambda,\beta}g\|_{1}<\|R_{\lambda}\|_{\infty\to1}$ for small $\beta$, since the quantity in Lemma \ref{lemmaipbd} is a formal derivative of $\|R_{\lambda,\beta}g\|_{1}$ with respect to $\beta$.

In this section, we formalize this argument, first by extending Lemma \ref{lemmaipbd} to a neighborhood of $\M$ in Lemma \ref{lem:pairing-stability}, then formalizing the derivative computation in Proposition \ref{prop:near-refined-fullP3} via Lemma \ref{lem:signflip3}.  Our final desired bound then appears in \eqref{upperconstants}, where we plug in some explicit values of the parameters.





For any $f\colon\R^{n}\to[-1,1]$ write $\mathrm{dist}_1(f,\mathcal M)\colonequals\inf_{g\in\mathcal M}\|f-g\|_1$.

Assume the following uniform constants exist (and do not depend on the dimension $n$):
\begin{align}
L&\colonequals\|R_\lambda\|_{2\to2},\label{eq:L}\\
\kappa_0 &\colonequals \inf_{g\in\mathcal M}\ \Big\langle \sgn(R_\lambda g),\,P_3 g\Big\rangle \ >\ 0,
\label{eq:kappa0}\\
K_0 &\colonequals \sup_{g\in\M}\ \|P_3 g\|_2 \ <\ \infty,
\label{eq:K0}\\
L_0 &\colonequals \sup_{g\in\mathcal M}\ \sup_{t\in(0,1]}\ \frac{\gamma_n(\{|R_\lambda g|\le t\})}{t}\ <\ \infty.
\label{eq:L0}
\end{align}
($L_{0}<\infty$ by Lemmas \ref{mombd}, \ref{lem:maximizers-are-pm1} and \ref{lem:smallball-P1} below, and $\kappa_{0}>0$ by Lemma \ref{lemmaipbd}.)

\begin{lemma}[$L_{1}$ bounds for Projections]\label{l1bd}
Let $h\colon\R^{n}\to[-1,1]$ with $\|h\|_{1}\leq1/100$.  Then
$$\|P_{1}h\|_{1}\leq\frac{1}{2}\|h\|_{1}\log(1/\|h\|_{1}).$$
$$\|P_{3}h\|_{1}\leq(e/\sqrt{3})^{3}\|h\|_{1}[\log(1/\|h\|_{1})]^{3/2}.$$
Also, if $h\colon\R^{n}\to[-1,1]$ (with no other assumptions), then $\|P_{3}h\|_{1}\leq 1$.
\end{lemma}
\begin{proof}
Suppose $\|h\|_{1}$ is fixed.  A rearrangement argument shows that the first inequality is saturated when $h$ is the indicator function of a half space, so we may assume $n=1$ and $h=1_{[a,\infty]}$ with $a>2.3$, in which case $\|h\|_{1}=\int_{a}^{\infty}\gamma_{1}(z)dz$ and $\|P_{1}h\|_{1}=\int_{\R}|z|\gamma_{1}(z)dz\cdot\int_{a}^{\infty}z\gamma_{1}(z)dz
=\sqrt{2/\pi}\gamma_{1}(a)$.  So, using $\gamma_{1}(a)\leq .583\Phi(-a)\log(1/\Phi(-a))$ $\forall$ $a\geq2.3$, we get
$$
\|P_{1}h\|_{1}=\sqrt{2/\pi}\gamma_{1}(a)\leq \sqrt{2/\pi}(.583) \int_{a}^{\infty}\gamma_{1}(z)dz\cdot\log\Big(1/\int_{a}^{\infty}\gamma_{1}(z)dz\Big).
$$

For the $P_{3}$ bound, we let $g\colon\R^{n}\to[-1,1]$, let $r,s>1$ with $1/r+1/s=1$ and $r\colonequals1+\frac{1}{3}\log(1/\|h\|_{1})>2$ and use H\"{o}lder's inequality to write
\begin{equation}\label{gp3bd}
|\langle g,P_{3}h\rangle|
=|\langle P_{3}g,h\rangle|
\leq\|P_{3}g\|_{r}\|h\|_{s}.
\end{equation}
By hypercontractivity for the $P_{3}$ term \cite{gross75}, for every $r\ge 2$, $\|P_{3}g\|_r\le (r-1)^{3/2}\|P_{3}g\|_2\leq(r-1)^{3/2}\|g\|_{2}\leq(r-1)^{3/2}$.  Since $|h|\leq1$ and $s\geq1$, we have $\|h\|_{s}^{s}\leq\|h\|_{1}$, so $\|h\|_{s}\leq\|h\|_{1}^{1/s}$.  In summary, (after taking the supremum over all $g\colon\R^{n}\to[-1,1]$ in \eqref{gp3bd}),
$$\|P_{3}h\|_{1}\leq\|h\|_{1}^{1/s}(r-1)^{3/2}
=\|h\|_{1}^{1-1/r}(r-1)^{3/2}.$$
Now, by definition of $r$, since $\|h\|_{1}<1/100$, we have $r>2$ and
$$
\|h\|_{1}^{-1/r}
=\exp\Big(\frac{\log(1/\|h\|_{1})}{1+\frac{1}{3}\log(1/\|h\|_{1})}\Big)\leq e^{3}.
$$
So, combining the above,
$$
\|P_{3}h\|_{1}
\leq \|h\|_{1}e^{3}((1/3)\log(1/\|h\|_{1}))^{3/2}.
$$
For the final assertion, note that $\|P_{3}h\|_{1}\leq\|P_{3}h\|_{2}\leq\|h\|_{2}\leq1$.
\end{proof}

For any integer $k\geq1$, denote $\mathcal{H}_k$ as the image of $P_{k}$ in $L_{2}(\R^{n},\gamma_{n})$.

\begin{lemma}[Third chaos on a bounded $Z$-strip]\label{lem:strip-H3}
Let $h\in\mathcal H_3$.  Fix $u\in S^{n-1}$ and write $X=Zu+Y$ where $Z=\langle u,X\rangle\sim N(0,1)$ and $u\perp Y$.
Then for every $z_0>0$ and every measurable $A\subseteq\{|Z|\le z_0\}$,
\begin{equation}\label{eq:strip-L1}
\int_{\R^{n}} |h(x)|\,1_A(x)\,\gamma_{n}(x)dx
\ \le\
\sqrt{C_{z_0}}\ \|h\|_2\ \gamma_n(A),
\end{equation}
where
\begin{equation}\label{eq:Cz0}
C_{z_0}\colonequals\sup_{|z|\le z_0}\Big(\frac{H_3(z)^2}{6}+\frac{H_2(z)^2}{2}+z^2+1\Big),
\qquad H_2(z)=z^2-1,\quad H_3(z)=z^3-3z.
\end{equation}
\end{lemma}

\begin{proof}
Decompose the third chaos in $(Z,Y)$ as
\[
h(Z,Y)=aH_3(Z)+H_2(Z)L(Y)+ZQ(Y)+C(Y),
\]
with $a\in\R$, $L\in\mathcal H_1(Y)$, $Q\in\mathcal H_2(Y)$, $C\in\mathcal H_3(Y)$ and the four summands are orthogonal (conditional on $Z=z$).
Then for each $z\in\R$,
\[
\E[h^2\mid Z=z]=a^2H_3(z)^2+H_2(z)^2\E[L^2]+z^2\E[Q^2]+\E[C^2]
\le \Big(\frac{H_3(z)^2}{6}+\frac{H_2(z)^2}{2}+z^2+1\Big)\|h\|_2^2.
\]
Thus on $A\subseteq\{|Z|\le z_0\}$,
\[
\int_A (h(x))^{2}\,\gamma_{n}(x)dx = \int_A \E[h^2\mid Z]\,\gamma_{n}(x)dx \le C_{z_0}\|h\|_2^2\,\gamma_n(A),
\]
and \eqref{eq:strip-L1} follows by the Cauchy--Schwarz inequality.
\end{proof}


\begin{lemma}[Tail bound for $\mathcal H_3$]\label{lem:H3-tail2}
Let $h\in\mathcal H_3$ with $\|h\|_2\le 1$. Then for every $s\ge e$,
\begin{equation}\label{eq:H3-tail}
\int_{\R^{n}} |h(x)|\,1_{\{|h|\ge s\}}\,\gamma_{n}(x)dx
\ \le\
\exp\Big(-\frac12(s/e)^{2/3} - \frac{1}{2}\Big).
\end{equation}
\end{lemma}

\begin{proof}
By hypercontractivity for $\mathcal{H}_{3}$ \cite{gross75}, for every $r\ge 2$, $\|h\|_r\le (r-1)^{3/2}\|h\|_2\le (r-1)^{3/2}$.
By Markov's inequality, $\gamma_n(|h|\ge s)\le \|h\|_r^r/s^r$.
Choose $r\colonequals1+(s/e)^{2/3}\ge 2$, so $(r-1)^{3/2}=s/e$, giving $\gamma_n(|h|\ge s)\le e^{-r}=\exp(-(s/e)^{2/3} -1)$.
Then by the Cauchy-Schwarz inequality, $\int_{\R^{n}} |h(x)|1_{\{|h|\ge s\}}\gamma_{n}(x)dx\le \|h\|_2\,\gamma_n(|h|\ge s)^{1/2}\le \exp(-\tfrac12(s/e)^{2/3}-1/2)$.
\end{proof}


\begin{lemma}[Small-ball via the $P_1$-direction]\label{lem:smallball-P1}
Let $f\colon\R^n\to\{-1,1\}$ and write $P_1 f(x)=\alpha_f\langle u_f,x\rangle$ with $\alpha_f=\|P_1 f\|_2$, $u_{f}\in S^{n-1}$.
Set $Z_f\colonequals\langle u_f,X\rangle\sim N(0,1)$ and $\eta_f\colonequals\lambda/\alpha_f$.
Then
\[
\{|R_\lambda f|\le t\}\subseteq \Big\{\big||Z_f|-\eta_f\big|\le \frac{t}{\alpha_f}\Big\},\qquad\forall\,t>0.
\]
Consequently, if $\alpha_f\ge \alpha_{\min}\colonequals.6>0$ and $t>0$, then
\begin{equation}\label{eq:smallball-linear-f}
\gamma_n(|R_\lambda f|\le t)\le \frac{4\gamma_{1}(0)}{\alpha_{\min}}\,t=\frac{4}{\alpha_{\min}\sqrt{2\pi}}\,t.
\end{equation}
\end{lemma}

\begin{proof}
Pointwise, $|R_\lambda f|\stackrel{\eqref{rdef}}{=}|P_1 f-\lambda f|\ge \big||P_1 f|-|\lambda f|\big|= \big||P_1 f|-\lambda\big|$ since $|f|= 1$.
With $|P_1 f|=\alpha_f|Z_f|$, this is $\ge \alpha_f\big||Z_f|-\eta_f\big|$, proving the inclusion.
The probability bound follows by the exact formula
$\P(||Z_{f}|-\eta|\le \delta)=2(\Phi(\eta+\delta)-\Phi(\eta-\delta))\le 4\gamma_{1}(0)\delta$ with $\delta=t/\alpha_{f}$ and $0<t<\lambda$, which implies $\delta<\eta$.  (If $t>\lambda$, then $\delta>\eta$ and $\P(||Z_{f}|-\eta|\leq\delta)=\P(|Z_{f}|\leq\eta+\delta)=\Phi(\eta+\delta)-\Phi(-\eta-\delta)\leq 4\gamma_{1}(0)\delta$.)
\end{proof}


\begin{lemma}[Flip correction bound with a free cutoff $t$]\label{lem:flip-correction-t}
Let $f\colon\R^n\to\{-1,1\}$ and set $a\colonequals R_\lambda f$ and $b\colonequals P_3 f\in\mathcal H_3$.
Assume $\alpha_f=\|P_1 f\|_2\ge \alpha_{\min}>0$.
Fix $\beta\in(0,1)$ and choose any cutoff $t\in(0,\lambda)$.
Assume also that $\{|a|\le t\}\subseteq \{|Z_f|\le z_0\}$ for some $z_0>0$
(e.g.\ it suffices by the triangle inequality that $z_0\ge \eta_f+t/\alpha_{\min}$).  
%
%
Assume $t/\beta>e$.  Then
\begin{equation}\label{eq:flip-correction-t}
\int_{\R^{n}} |b(x)|\,1_{\{|a|\le \beta|b|\}}\,\gamma_{n}(x)dx
\ \le\
\sqrt{C_{z_0}}\ \|b\|_2\ \gamma_n(|a|\le t)\;+\;\exp\Big(-\frac12\big(\tfrac{t}{e\beta}\big)^{2/3}-\frac{1}{2}\Big),
\end{equation}
where $C_{z_0}$ is as in \eqref{eq:Cz0}. In particular, since $\|b\|_2\le \|f\|_2\le 1$,
combining with \eqref{eq:smallball-linear-f} yields
\begin{equation}\label{eq:flip-correction-final}
2\beta\int_{\R^{n}} |P_3 f(x)|\,1_{\{|R_\lambda f|\le \beta|P_3 f|\}}\,\gamma_{n}(x)dx
\ \le\
\underbrace{\frac{8\sqrt{C_{z_0}}}{\alpha_{\min}\sqrt{2\pi}}}_{=:K_{\mathrm{strip}}}\ \beta t
\;+\;2\beta\exp\Big(-\frac12\big(\tfrac{t}{e\beta}\big)^{2/3}-\frac{1}{2}\Big).
\end{equation}
\end{lemma}

\begin{proof}
Split $\{|a|\le \beta|b|\}\subseteq \{|a|\le t\}\cup \{|b|\ge t/\beta\}$, hence
\[
\int_{\R^{n}} |b(x)|1_{\{|a|\le \beta|b|\}}\gamma_{n}(x)dx
\le
\int_{\R^{n}} |b(x)|1_{\{|a|\le t\}}\gamma_{n}(x)dx
+\int_{\R^{n}} |b(x)|1_{\{|b|\ge t/\beta\}}\gamma_{n}(x)dx.
\]
On $\{|a|\le t\}\subseteq \{|Z_f|\le z_0\}$ apply Lemmas~\ref{lem:strip-H3} and~\ref{lem:H3-tail2} to get \eqref{eq:flip-correction-t} with $s=t/\beta$.
To obtain \eqref{eq:flip-correction-final}, apply Lemma~\ref{lem:smallball-P1} to \eqref{eq:flip-correction-t}. 
\end{proof}



We now extend Lemma \ref{lemmaipbd} to a neighborhood of $\M$.

\begin{lemma}[Pairing stability near maximizers]\label{lem:pairing-stability}
Let $f\colon\R^n\to[-1,1]$ with $\mathrm{dist}_1(f,\mathcal M)\le \varepsilon$.
Pick $g\in\mathcal M$ with $\|f-g\|_1\le \varepsilon<1/100$ and define $h\colonequals f-g$.
Let $s_f\in\mathrm{sign}(R_\lambda f)$.
Then
\begin{equation}\label{eq:pairing-stability}
\langle s_f,P_3 f\rangle
\ \ge\
\kappa_0-3.87\epsilon[\log(2/\epsilon)]^{3/2}
-2^{3/2}[\epsilon L_{0}(\lambda+.5\log(2/\epsilon))]^{1/4}K_{0}.
\end{equation}
\end{lemma}
\begin{proof}
Fix $f$ with $\mathrm{dist}_1(f,\mathcal M)\le\varepsilon$ and pick $g\in\mathcal M$ with 
\begin{equation}\label{epsdef}
\|f-g\|_1\le\varepsilon.
\end{equation}
Write
\[
s_f\in \mathrm{sign}(R_\lambda f),
\qquad
s_g\in \mathrm{sign}(R_\lambda g),
\]
i.e.\ $s_f R_\lambda f=|R_\lambda f|$ a.e.\ and similarly for $s_g$.

\smallskip
\noindent\textbf{Step 1: sign stability $\|s_f-s_g\|_2\leq \widetilde{O}(\varepsilon^{1/4})$.}

Recall $h\colonequals f-g$. Then $R_\lambda f=R_\lambda g+R_\lambda h$.
If $\mathrm{sign}(a+b)\neq \mathrm{sign}(a)$ then $|a|\le |b|$, so using $a=R_{\lambda}g$ and $b=R_{\lambda}h$ and $s_{f}=\mathrm{sign}(a+b)$, for any $u>0$,
\[
\{s_f\neq s_g\}
\subseteq \{|R_\lambda g|\le |R_{\lambda} h|\}
\subseteq\{|R_{\lambda}g|\leq u\}\cup\{|R_{\lambda}h|\geq u\}.
\]
Therefore, for any $u>0$,
\begin{equation}\label{gbd}
\gamma_n(s_f\neq s_g)
\le \gamma_n(|R_\lambda g|\le u)+\gamma_n(|R_\lambda h|\ge u).
\end{equation}
By \eqref{eq:L0}, $\gamma_n(|R_\lambda g|\le u)\le L_0 u$.


By Lemma \ref{l1bd} for $h/2$, $0<\epsilon<1/100$, and then Markov's inequality and $\|h\|_{1}\leq\epsilon$,  
$$\|R_{\lambda}h\|_{1}\stackrel{\eqref{rdef}}{\leq}\lambda\|h\|_{1}+\|P_{1} h\|_{1}\leq\lambda\|h\|_{1}+\frac{1}{2}\|h\|_{1}\log(2/\|h\|_{1}).$$
%
%
\[
\gamma_n(\{|R_\lambda h|\ge u\})\le \frac{\|R_\lambda h\|_1}{u}\le \frac{\epsilon(\lambda+.5\log(2/\epsilon))}{u}.
\]
Minimizing $L_0u+c/u$ over $u>0$ in 
\eqref{gbd} gives (with $c=\epsilon(\lambda+.5\log(2/\epsilon))$)
%
%
%
%
%
%
%

\[
\gamma_n(\{s_f\neq s_g\})\stackrel{\eqref{gbd}}{\le} 2\sqrt{L_{0}\epsilon(\lambda+.5\log(2/\epsilon))}.
\]

Since $|s_f-s_g|=2$ on $\{s_f\neq s_g\}$ and $0$ otherwise,
\begin{equation}\label{sfineq}
\|s_f-s_g\|_2 = 2[\gamma_n(s_f\neq s_g)]^{1/2}
\leq 2^{3/2}[\epsilon L_{0}(\lambda+.5\log(2/\epsilon))]^{1/4}.
\end{equation}
\smallskip
\noindent\textbf{Step 2: pairing stability $\langle s_f,P_3 f\rangle \ge \kappa_0 - \widetilde{O}(\epsilon^{1/4})$}

Using $P_3 f=P_3 g+P_3 h$,
\begin{equation}\label{nine77}
\langle s_f,P_3 f\rangle
=
\langle s_g,P_3 g\rangle
+\langle s_f-s_g,P_3 g\rangle
+\langle s_f,P_3 h\rangle.
\end{equation}
The first term is $\ge \kappa_0$ by \eqref{eq:kappa0}.
For the second term, we apply Cauchy-Schwarz to get
%
\[
|\langle s_f-s_g,P_3 g\rangle|
\le \|s_f-s_g\|_{2}\,\|P_3 g\|_2
\stackrel{\eqref{eq:K0}\wedge\eqref{sfineq}}{\le} 
2^{3/2}[\epsilon L_{0}(\lambda+.5\log(2/\epsilon))]^{1/4}K_{0}.
\]

%
%


For the third term, Lemma \ref{l1bd} for $h/2$ says $$\|P_3 h\|_1 \leq(e/\sqrt{3})^{3}\|h\|_1[\log(2/\|h\|_{1})]^{3/2}\stackrel{\eqref{epsdef}}{\le} 3.87\epsilon[\log(2/\epsilon)]^{3/2},$$
so
\[
|\langle s_f,P_3 h\rangle|\le \|P_3 h\|_1\le 3.87\epsilon[\log(2/\epsilon)]^{3/2}.
\]
Hence, combining the above
\begin{equation}\label{eq:pairing-lower}
\langle s_f,P_3 f\rangle \ \stackrel{\eqref{nine77}}{\ge}\ \kappa_0-3.87\epsilon[\log(2/\epsilon)]^{3/2}
-2^{3/2}[\epsilon L_{0}(\lambda+.5\log(2/\epsilon))]^{1/4}K_{0}.
\end{equation}

\end{proof}




We can now turn Lemma \ref{lem:pairing-stability} into an operator norm upper bound.

\begin{prop}\label{prop:near-refined-fullP3}
Let $\varepsilon\in(0,1)$ and suppose $f\colon\R^n\to[-1,1]$ satisfies
$\mathrm{dist}_1(f,\mathcal M)\le \varepsilon$.
Assume moreover that $\alpha_f=\|P_1 f\|_2\ge \alpha_{\min}>0$.  Then for every $\beta\in(0,1)$ and every cutoff $t\in(0,\lambda)$ such that
$\{|R_\lambda f|\le t\}\subseteq\{|Z_f|\le z_0\}$ (for some $z_0$, such as $z_{0}>\eta_{f} + t/\alpha_{\rm min}$),
\begin{equation}\label{eq:near-refined-t}
\begin{aligned}
\|R_{\lambda,\beta}f\|_1
&\ \le\
\|R_\lambda\|_{\infty\to1}
-\beta\Big(\kappa_0-3.87\epsilon[\log(2/\epsilon)]^{3/2}
-2^{3/2}[\epsilon L_{0}(\lambda+.5\log(2/\epsilon))]^{1/4}K_{0}\Big)\\
&+K_{\mathrm{strip}}\beta t
+2\beta\exp\Big(-\frac12\big(\tfrac{t}{e\beta}\big)^{2/3}-\frac{1}{2}\Big),
\end{aligned}
\end{equation}
Here $K_{\mathrm{strip}}=\frac{8\sqrt{C_{z_0}}}{\alpha_{\min}\sqrt{2\pi}}$.  In particular, fix any exponent $\rho\in(0,1)$ and choose $t=\beta^\rho$. Then
\begin{equation}\label{eq:near-refined-rho}
\|R_{\lambda,\beta}f\|_1
\ \le\
\|R_\lambda\|_{\infty\to1}
-\beta\,\kappa_{\mathrm{eff}}(\varepsilon)
+K_{\mathrm{strip}}\,\beta^{1+\rho}
+2\beta\exp\Big(-\frac12e^{-2/3}\,\beta^{-\frac{2}{3}(1-\rho)}-\frac{1}{2}\Big),
\end{equation}
where
\[
\kappa_{\mathrm{eff}}(\varepsilon)\colonequals\kappa_0-3.87\epsilon[\log(2/\epsilon)]^{3/2}
-2^{3/2}[\epsilon L_{0}(\lambda+.5\log(2/\epsilon))]^{1/4}K_{0}.
\]

\end{prop}
Consequently, if $\kappa_{\mathrm{eff}}(\varepsilon)>0$, then there exists $\beta_0(\varepsilon,\rho)>0$
such that for all $0<\beta\le \beta_0(\varepsilon,\rho)$,
\[
\|R_{\lambda,\beta}f\|_1 \le \|R_\lambda\|_{\infty\to1}-\tfrac12\,\kappa_{\mathrm{eff}}(\varepsilon)\,\beta.
\]

Thus: \emph{for every fixed $\varepsilon$ with $\kappa_{\mathrm{eff}}(\varepsilon)>0$, the entire $\varepsilon$-neighborhood
of $\mathcal M$ experiences a uniform linear-in-$\beta$ drop for all sufficiently small $\beta$, with $\varepsilon$ independent of $\beta$.}
\begin{proof}
Apply Lemma~\ref{lem:signflip3} with $a=R_\lambda f$ and $b=P_3 f$, integrate, and use $\|R_\lambda f\|_1\le\|R_\lambda\|_{\infty\to1}$:
\[
\|R_{\lambda,\beta}f\|_1
\stackrel{\eqref{rlbdef}}{\leq}
\|R_\lambda\|_{\infty\to1}
-\beta\langle s_f,P_3 f\rangle
+2\beta\int_{\R^{n}} |P_3 f(x)|1_{\{|R_\lambda f|\le \beta|P_3 f|\}}\gamma_{n}(x)dx.
\]
Lower bound $\langle s_f,P_3 f\rangle$ by Lemma~\ref{lem:pairing-stability}.
Upper bound the last integral term by Lemma~\ref{lem:flip-correction-t}.
This yields \eqref{eq:near-refined-t}. The specialization \eqref{eq:near-refined-rho} follows by choosing $t=\beta^\rho$.
\end{proof}


\begin{remark}


Lemma \ref{mombd} says if $\epsilon<1/100$, we can take $\alpha_{\text{min}}=.6$, $\lambda=.19747\ldots$, $\eta_{f}=\lambda/\alpha_{f}<1/3$, $t=\beta^{\rho}$, $z_{0}\colonequals 1/3+ t/.6=1/3+\beta^{\rho}/.6$, $C_{z_{0}}\stackrel{\eqref{eq:Cz0}}{\leq}1.7$ (if $z_{0}<.36$), $K_{\text{strip}}=\frac{8\sqrt{C_{z_{0}}}}{\alpha_{\text{min}}\sqrt{2\pi}}=\frac{8}{.6\sqrt{\pi}}\sqrt{\frac{1.7}{2}}\leq7$.  From Lemma \ref{lemmaipbd}, we can take $\kappa_{0}\stackrel{\eqref{eq:kappa0}}{=}.0454$, $K_{0}\leq\sqrt{.0871+.04141}\leq.359$ (by Lemma \ref{lemmaipbd} modified to get an upper bound instead of a lower bound, i.e. $\|P_{3}g\|_{2}^{2}\leq(B+A)^{2}/6+p^{2}+s_{1}^{2}+t_{2}^{2}/2$ and $(B+A)^{2}\leq.7226^{2}$), $L_{0}\leq4/[\alpha_{min}\sqrt{2\pi}]\leq2.66$ by \eqref{eq:L0} and \eqref{eq:smallball-linear-f}.
Then we can take $$\kappa_{\mathrm{eff}}(\epsilon)=.0454 - 3.87\epsilon[\log(2/\epsilon)]^{3/2} -\epsilon^{1/4}2^{3/2}(2.66)^{1/4}(.19747+.5\log(2/\epsilon))^{1/4}(.359).$$

Choose $\epsilon=10^{-7}$, then $\kappa_{\mathrm{eff}}\geq.0454-.0396>.0058$, choose $\rho=.7$, and $\beta \leq 10^{-10}$.  Then $K_{\mathrm{strip}}\beta^{\rho}\leq10^{-6}$ and $2\exp(-.5-.5e^{-2/3}\beta^{-(2/3)(1-\rho)})<10^{-10}$, $z_{0}<.36$, so \eqref{eq:near-refined-rho} says
\begin{equation}\label{upperconstants}
\begin{aligned}
\|R_{\lambda,\beta}f\|_{1}
&\leq\|R_{\lambda}\|_{\infty\to1}
-\beta(.0058 - 10^{-6} - 10^{-10})\\
&\leq\|R_{\lambda}\|_{\infty\to1}
-\beta(.0057),\\
&\qquad\forall\,f\colon\R^{n}\to\{-1,1\}\text{ with }\inf_{g\in\M}\|f-g\|_{1}<\epsilon,\forall\,0<\beta<10^{-10}.
\end{aligned}
\end{equation}



\end{remark}

\begin{remark}
Theorem \ref{mainthm} with an ineffective constant follows immediately from \eqref{upperconstants}, since there must exist some $c(\epsilon)$ such that $\|R_{\lambda}f\|_{1}<\|R_{\lambda}\|_{\infty\to1} - c(\epsilon)$ for all $f\colon\R^{n}\to[-1,1]$ with $\inf_{g\in\M}\|f-g\|_{1}>\epsilon$, hence $\|R_{\lambda,\beta}f\|_{1}<\|R_{\lambda}\|_{\infty\to1} - c(\epsilon)+\beta$, for all such $f$ (using the final part of Lemma \ref{l1bd}).  That is, choosing $0<\beta<\min(c(\epsilon),10^{-10})$ (and using \eqref{upperconstants}) completes the proof of Theorem \ref{mainthm} (with an ineffective constant), i.e. $\|R_{\lambda,\beta}\|_{\infty\to1}<\|R_{\lambda}\|_{\infty\to1}$.  Since we would like an effective constant, we now proceed to find an explicit form for $c(\epsilon)$.
\end{remark}

\section{Lower Bounds for the Operator Norm}\label{seceight}

Having obtained a norm drop for $R_{\lambda,\beta}$ in \eqref{upperconstants} for a fixed $L_{1}$ neighborhood of $\M$, we now move on to the second step of the main theorem's proof, i.e. proving a norm drop for $\|R_{\lambda}f\|_{1}$ for all $f$ outside an $\epsilon$ neighborhood of $\M$.


We first recall notation from Section \ref{sectwo}.  Fix $0<\lambda_* <1$ and let $0<\eta_*<1$ solve
\[
\lambda_* = 2\eta_*\gamma_{1}(\eta_*).
\]
For any $\lambda,\alpha>0$, define
\begin{equation}\label{reedseq}
\eta\colonequals\frac{\lambda}{\alpha},
\qquad
\eta_* \colonequals \frac{\lambda_*}{\alpha_*}
\qquad \alpha_*= 2\gamma_{1}(\eta_*).
\end{equation}
(These are the usual Reeds/Davie parameters; numerically $\lambda_*\approx 0.197479091$, $\eta_*\approx 0.25573$,
$\alpha_*\approx 0.7722165$.)  Let $f\colon\R^n\to[-1,1]$ satisfy $\|P_1 f\|_2=\alpha$. Let $u\in S^{n-1}$ be such that
\[
(P_1 f)(x)=\alpha\langle u,x\rangle,\qquad\forall\,x\in\R^{n},
\]
and define $Z\colonequals\langle u,X\rangle\sim N(0,1)$ for $X\sim\gamma_n$.
Define the conditional profiles
\begin{equation}\label{thetadef}
\theta(z)\colonequals\E[f(X)\mid Z=z]\in[-1,1],
\qquad
\nu(z)\colonequals\E[f(X)^2\mid Z=z]\in[0,1].
\end{equation}
Assume $\theta$ is odd (we will show later there is no loss of generality in assuming this).  For any $\alpha,\lambda>0,z\in\R$, define

\begin{equation}\label{abdefs}
A(z)=A_{\alpha}(z)\colonequals\frac{|\alpha z-\lambda|+|\alpha z+\lambda|}{2},
\qquad
B(z)=B_{\alpha}(z)\colonequals\frac{|\alpha z-\lambda|-|\alpha z+\lambda|}{2}.
\end{equation}
For $z>0$, one has the explicit form
\begin{equation}\label{eq:B-piecewise}
B(z)=
\begin{cases}
-\alpha z,& 0<z<\eta,\\
-\lambda,& z>\eta.
\end{cases}
\end{equation}

Define, for $\xi:\R\to[-1,1]$ satisfying the moment constraint $\int_{\R} z\xi(z)\gamma_{1}(z)\,dz=\alpha$,
\begin{equation}\label{vdef}
V_{\alpha,\lambda}(\xi)\colonequals\int_{\R}\gamma_{1}(z)\big(A_{\alpha}(z)+\xi(z)B_{\alpha}(z)\big)\,dz.
\end{equation}

Let
\begin{equation}\label{fstdef}
\begin{aligned}
F_{\alpha,\lambda}&\colonequals\sup\{V_{\alpha,\lambda}(\xi):\quad\xi\colon\R\to[-1,1], \ |\xi|\le 1,\ \int_{\R} z\xi(z)\gamma_{1}(z)dz=\alpha\}.\\
F_{*}&\colonequals\sup_{\alpha,\lambda>0}F_{\alpha,\lambda}.
\end{aligned}
\end{equation}

\begin{lemma}[Dual certificate gap formula]\label{lem:dual-gap}
For all $0<\alpha<1$ with $|\alpha-\alpha_*|<1/100$, for every feasible $\xi$ (i.e.\ $|\xi|\le 1$ and $\int_{\R} z\xi(z)\gamma_{1}(z)\,dz=\alpha$),
one has
\begin{equation}\label{eq:gap-formula}
F_{\alpha,\lambda}-V_{\alpha,\lambda}(\xi)
=
\int_{|z|>\eta}(\alpha |z|-\lambda)\,(1-\xi(z)\mathrm{sign}(z))\,\gamma_{1}(z)\,dz.
\end{equation}
\end{lemma}

\begin{proof}
\emph{Step 1: Weak duality.}
For $\mu\in\R$, define the dual functional
\begin{equation}\label{ddef}
D_{\alpha}(\mu)\colonequals\int_\R A(z)\gamma_{1}(z)dz\;+\;\mu\alpha\;+\;\int_\R |B(z)-\mu z|\gamma_{1}(z)dz.
\end{equation}
For any feasible $\xi$ ($|\xi|\le 1$, $\int_{\R} z\xi(z)\gamma_{1}(z)dz=\alpha$), we have
\begin{equation}\label{vdef2}
\begin{aligned}
V_{\alpha,\lambda}(\xi)
&\stackrel{\eqref{vdef}}{=}\int_{\R} A(z)\gamma_{1}(z)dz  + \int_{\R} \gamma_{1}(z)\xi(z) B(z)dz\\
&=\int_{\R} A(z)\gamma_{1}(z)dz + \mu\alpha + \int_{\R} \xi(z)(B(z)-\mu z)\gamma_{1}(z)dz,
\end{aligned}
\end{equation}
where we used the moment constraint to replace $\int_{\R} \xi(z)\mu z\gamma_{1}(z)dz$ by $\mu\alpha$.
Since $|\xi|\le 1$, pointwise $\xi(B-\mu z)\le |B-\mu z|$, hence $V_{\alpha,\lambda}(\xi)\le D_{\alpha}(\mu)$.
Taking the infimum over $\mu$ gives $V_{\alpha,\lambda}(\xi)\le \inf_{\mu\in\R} D_{\alpha}(\mu)$ and therefore
\begin{equation}\label{fsineq}
F_{\alpha,\lambda}\stackrel{\eqref{fstdef}}{=}\sup_{\substack{\xi\colon\R\to[-1,1]\\ \int_{\R}z\xi(z)\gamma_{1}(z)dz=\alpha}}V_{\alpha,\lambda}(\xi)\le \inf_{\mu\in\R}D_{\alpha}(\mu).
\end{equation}

\emph{Step 2: The Reeds dual optimizer $\mu=-\alpha$.}
Set $\mu\colonequals-\alpha$ and define the slack
\begin{equation}\label{sdefcap}
S(z)=S_{\alpha}(z)\colonequals B(z)-\mu z = B(z)+\alpha z,\qquad\forall\,z\in\R.
\end{equation}
Using \eqref{eq:B-piecewise}, for $z>0$ we have $S$ is odd and
\[
S(z)=
\begin{cases}
0,& 0<z<\eta,\\
\alpha z-\lambda=\alpha(z-\eta),& z>\eta.
\end{cases}
\]
By oddness, $S(z)<0$ for $z<-\eta$ and $S(z)=0$ for $|z|<\eta$.

\emph{Step 3: A primal point attaining $D_{\alpha}(\mu)$.}
Let $\xi\colon\R\to[-1,1]$ satisfy $\int_{\R}z\xi(z)\gamma_{1}(z)dz=\alpha$, and $\xi(z)=\mathrm{sign}(z)$ for all $|z|>\eta$.  (Since $|\alpha-\alpha_*|<1/100$, such a $\xi$ exists by e.g. defining $\xi(z)$ to be $a\cdot\mathrm{sign}(z)$ for all $|z|<\eta$ for some $a\in[-1,1]$ and $\xi(z)=\mathrm{sign}(z)$ for all $|z|>\eta$, so that $\int_{\R}z\xi(z)\gamma_{1}(z)dz=2\gamma_{1}(\eta)+2a(\gamma_{1}(0)-\gamma_{1}(\eta))$.  When $a=0$ and $\alpha=\alpha_*$, this integral is $2\gamma_{1}(\eta_*)=\alpha_*$.  Choosing $a=-1$ or $a=1$ shows that any intermediate value in $[\alpha_*-1/100,\alpha_*+1/100]$ can be achieved, using also e.g. $|\eta-\eta_*|<1/100$ by \eqref{etastin} below; $a=1$ gives $\sqrt{2/\pi}>\alpha_*+.01$; when $a=-1$, $\eta\geq\eta_*-.01\geq.24$ and the integral is $\leq2\gamma_{1}(.24)-2(\gamma_{1}(0)-\gamma_{1}(.24))\leq.753<\alpha_*-.01$.)
Then $\xi(z)=\mathrm{sign}(S(z))$ for all $|z|>\eta$, and $S(z)=0$ for all $|z|<\eta$, so $\xi(z)S(z)=|S(z)|$ for all $z\in\R$, and
\begin{flalign*}
V_{\alpha,\lambda}(\xi)
&\stackrel{\eqref{vdef2}\wedge\eqref{sdefcap}}{=}\int_{\R}A(z)\gamma_{1}(z)dz + \mu\alpha + \int_{\R} \xi(z) S(z)\gamma_{1}(z)dz\\
&=\int_{\R}A(z)\gamma_{1}(z)dz + \mu\alpha + \int_{\R} |S(z)|\gamma_{1}(z)dz
\stackrel{\eqref{ddef}\wedge\eqref{sdefcap}}{=} D_{\alpha}(\mu).
\end{flalign*}
Hence $F_{\alpha,\lambda}\stackrel{\eqref{fstdef}}{\ge} V_{\alpha,\lambda}(\xi)=D_{\alpha}(\mu)$. Combined with \eqref{fsineq} ($F_{\alpha,\lambda}\le \inf_{\mu\in\R} D_{\alpha}(\mu)$),
we get
$$F_{\alpha,\lambda}=D_{\alpha}(\mu).$$

\emph{Step 4: Gap identity.}
For any feasible $\xi$,
\[
F_{\alpha,\lambda}-V_{\alpha,\lambda}(\xi)=D_{\alpha}(\mu)-V_{\alpha,\lambda}(\xi)
\stackrel{\eqref{ddef}\wedge\eqref{vdef2}\wedge\eqref{sdefcap}}{=}\int_{\R} \Big(|S(z)|-\xi(z)S(z)\Big)\gamma_{1}(z)dz.
\]
On $|z|<\eta$, $S(z)=0$ so the integrand vanishes.
On $z>\eta$, $S(z)=\alpha z-\lambda>0$, so $|S(z)|-\xi S(z)=(1-\xi(z))(\alpha z-\lambda)$.  Similarly, on $z<-\eta$, $S(z)=\lambda-\alpha|z|$, $|S(z)|=\alpha|z|-\lambda$, so $|S(z)|-\xi(z)S(z)=(\alpha|z|-\lambda)(1-\mathrm{sign}(z)\xi(z))$.
\end{proof}

\section{Lower Bounds and L1 Distance}\label{secnine}

In the previous section, we showed that when $|\theta|\neq1$ for $|z|>\eta$, $\theta=\theta_{f}$ does not optimize $\|R_{\lambda}f\|_{1}$.  In this section, we upgrade this statement to: if $\theta$ is far from the optimal set $\Theta$ in the $L_{1}$ norm, then $\theta$ does not optimize $\|R_{\lambda}f\|_{1}$, in Lemmas \ref{tailv2} and \ref{lem:Delta-large-gap}.  

\begin{equation}\label{deldef}
\Delta=\Delta_{\eta_{*}}\colonequals\left|\int_{-\eta_{*}}^{\eta_{*}} z\,\theta(z)\gamma_{1}(z)\,dz\right|.
\end{equation}

$$
d\colonequals\inf_{\xi\in\Theta}\|\theta-\xi\|_{1}.
$$


\begin{lemma}\label{lower1}
If $\Delta\leq\eta_{*} d/4$ and if $\theta\colon\R\to[-1,1]$, then

\begin{equation}\label{eq:tail-defect}
\int_{|z|>\eta_{*}}\bigl(1-\theta(z)\mathrm{sign}(z)\bigr)\gamma_{1}(z)\,dz \ \ge\ \frac{1}{2}\,d.
\end{equation}
\end{lemma}

\begin{proof}
We explicitly construct $\xi\in\Theta$ and bound $\|\theta-\xi\|_1$.

\textbf{Step 1: Fix the tail to match $\Theta$.}
Define $\overline{\theta}\colon\R\to\R$ by setting $\overline{\theta}(z)=\mathrm{sign}(z)$ for $|z|>\eta_{*}$ and $\overline{\theta}(z)=\theta(z)$ for $|z|\le \eta_{*}$.
Then $|\overline{\theta}|\le 1$, and $\overline{\theta}(z)=1$ for all $z>\eta_{*}$.
Moreover, on the tail $|z|>\eta_{*}$ we have $\overline{\theta}(z)-\theta(z)=\mathrm{sign}(z)-\theta(z)$, so
\begin{equation}\label{eq:tail-L2}
\|\theta-\overline{\theta}\|_1
=
\int_{|z|>\eta_{*}}(1-\theta(z)\mathrm{sign}(z))\gamma_{1}(z)\,dz.
\end{equation}

\textbf{Step 2: Correct the inner moment by modifying on a set $S\subset[\eta_{*}/2,\eta_{*}]$.}
Let
\begin{equation}\label{del0def}
\Delta'\colonequals\int_0^{\eta_{*}} z\,\overline{\theta}(z)\gamma_{1}(z)\,dz=\int_0^{\eta_{*}} z\,\theta(z)\gamma_{1}(z)\,dz,
\qquad \Delta=|\Delta'|.
\end{equation}
We want to change $\overline{\theta}$ on $(-\eta_{*},\eta_{*})$ so that the new function $\xi$ satisfies
$\int_{-\eta_{*}}^{\eta_{*}} z\,\xi(z)\gamma_{1}(z)\,dz=0$.
Let $s\colonequals\mathrm{sign}(\Delta')\in\{-1,+1\}$ (with $s=+1$ if $\Delta'=0$).

Consider the nonnegative integrable function on $\eta_{*}/2<|z|<\eta_{*}$
\[
g(z)\colonequals z\bigl(\mathrm{sign}(z)+s\,\theta(z)\bigr)\gamma_{1}(z)\ \ge\ 0.
\]
We claim
\begin{equation}\label{eq:capacity}
\int_{\eta_{*}/2<|z|<\eta_{*}} g(z)\,dz \ \ge\ \Delta.
\end{equation}
Assuming this claim for a moment, define $G(t)\colonequals\int_{\eta_{*}/2<|z|<t} g(z)\,dz$ for $t\in[\eta_{*}/2,\eta_{*}]$.
Then $G$ is continuous, $G(\eta_{*}/2)=0$, and $G(\eta_{*})\ge \Delta$ by \eqref{eq:capacity}, so by the intermediate value theorem there exists
$t_0\in[\eta_{*}/2,\eta_{*}]$ such that $G(t_0)=\Delta$.
Set $S\colonequals\{z\in\R\colon \eta_{*}/2<|z|<t_{0}\}$; then
\begin{equation}\label{eq:choose-S}
\int_S z(\mathrm{sign}(z)+s\theta(z))\gamma_{1}(z)\,dz = \Delta.
\end{equation}

Now define $\xi$ on $(-\eta_{*},\eta_{*})$ by
\begin{equation}\label{xidef}
\xi(z)\colonequals
\begin{cases}
-s\cdot\mathrm{sign}(z),& z\in S,\\
\theta(z),& z\in (-\eta_*,\eta_{*})\setminus S,
\end{cases}
\end{equation}
and on the tails set $\xi(z)=\mathrm{sign}(z)$ for $|z|>\eta_{*}$.
Then $|\xi|\le 1$, and $\xi(z)=1$ for $z>\eta_{*}$.

Finally, we show $\int_{-\eta_{*}}^{\eta_{*}}\xi(z)z\gamma_{1}(z)dz=0$:
if $\Delta'>0$ then $s=+1$ and on $S$ we changed $\theta$ to $-1$, so the inner moment $\Delta'$ decreases by
$\int_S z(\theta(z)+\mathrm{sign}(z))\gamma_{1}(z)\,dz=\Delta$ (by \eqref{eq:choose-S}), hence $\int_{-\eta_{*}}^{\eta_{*}}\xi(z)z\gamma_{1}(z)dz=0$.
If $\Delta'<0$ then $s=-1$ and on $S$ we changed $\theta$ to $+1$, increasing the inner moment by
$\int_S z(\mathrm{sign}(z)-\theta(z))\gamma_{1}(z)\,dz=\Delta$, again making it $0$.
Thus $\xi\in\Theta$.

\textbf{Step 3: Bound the $L_{1}$ cost of the inner correction.}
On $S$ we have
\[
|\theta(z)-\xi(z)| \stackrel{\eqref{xidef}}{=} |\theta(z)-(-s)\mathrm{sign}(z)| =|s(\theta(z)+s\cdot\mathrm{sign}(z))|= 1+s\theta(z)\mathrm{sign}(z)\in[0,2],
\]
so
\begin{equation}\label{eq:inner-cost-start}
\int_S |\theta(z)-\xi(z)|\gamma_{1}(z)dz
=
\int_S (1+s\theta(z)\mathrm{sign}(z))\gamma_{1}(z)dz
\end{equation}

Now use that $|z|\ge \eta_{*}/2$ on $S\subset\{z\in\R\colon \eta_{*}/2<|z|<\eta_{*}\}$ and $1+s\theta\cdot\mathrm{sign}(z)\ge 0$ on $S$:
\[
\Delta\stackrel{\eqref{eq:choose-S}}{=}\int_S z(\mathrm{sign}(z)+s\theta(z))\gamma_{1}(z)dz \ \ge\ \frac{\eta_{*}}{2}\int_S (1+s\theta(z)\mathrm{sign}(z))\gamma_{1}(z)dz,
\]
Insert this into \eqref{eq:inner-cost-start} to get
\[
\int_S |\theta(z)-\xi(z)|\gamma_{1}(z)dz \le  \frac{2\Delta}{\eta_{*}}.
\]
Combining this with \eqref{xidef} gives
\begin{equation}\label{eq:inner-cost}
\int_{|z|<\eta_{*}}|\overline{\theta}(z)-\xi(z)|\gamma_{1}(z)dz
=
\int_{|z|<\eta_{*}}|\theta(z)-\xi(z)|\gamma_{1}(z)dz \le \frac{2\Delta}{\eta_{*}}.
\end{equation}

\textbf{Step 4: Combine tail + inner costs.}
Using $\|\theta-\xi\|_1\le \|\theta-\overline{\theta}\|_1 + \|\overline{\theta}-\xi\|_1$
(or directly summing disjoint supports: tail vs. inner strip), combining \eqref{eq:tail-L2} and \eqref{eq:inner-cost}
gives 

\begin{equation}\label{eq:dist-upper}
\|\theta-\xi\|_1
\ \le\
\int_{|z|>\eta_{*}}|1-\theta(z)\mathrm{sign}(z)|\gamma_{1}(z)\,dz
\;+\;\frac{2}{\eta_{*}}\,\Delta,
\end{equation}

\begin{equation}\label{eq:d2-upper}
d
\ \le\
\int_{|z|>\eta_{*}}|1-\theta(z)\mathrm{sign}(z)|\gamma_{1}(z)\,dz
\;+\;\frac{2}{\eta_{*}}\,\Delta,
\end{equation}

\textbf{Step 5: One-line consequence: $\Delta$ small forces tail defect.}
If $\Delta\le \frac{\eta_{*}}{4}d$, then \eqref{eq:d2-upper} implies
\[
\int_{|z|>\eta_{*}} |1-\theta(z)\mathrm{sign}(z)|\gamma_{1}(z)dz \ \ge\ d - \frac{2}{\eta_{*}}\Delta \ \ge\ \frac{d}{2},
\]
which is \eqref{eq:tail-defect}.

\emph{Proof of the capacity claim \eqref{eq:capacity}.}
It remains to prove $\int_{\eta_{*}/2<|z|<\eta_{*}} z(\mathrm{sign}(z)+s\theta(z))\gamma_{1}(z)dz \ge |\Delta'|$.
Assume first $\Delta'\ge 0$, so $s=+1$ and $\Delta=\Delta'$.
Then
\begin{flalign*}
\int_{\eta_{*}/2<|z|<\eta_{*}} z(\mathrm{sign}(z)+\theta(z))\gamma_{1}(z)dz
&=
\int_{\eta_{*}/2<|z|<\eta_{*}} z\theta(z)\gamma_{1}(z)dz + \int_{\eta_{*}/2<|z|<\eta_{*}} |z|\gamma_{1}(z)dz\\
&\stackrel{\eqref{del0def}}{=}
\Delta' - \int_{-\eta_{*}/2}^{\eta_{*}/2} z\theta(z)\gamma_{1}(z)dz + \int_{\eta_{*}/2<|z|<\eta_{*}} |z|\gamma_{1}(z)dz.
\end{flalign*}
Since $\theta\le 1$, we have $\int_{-\eta_{*}/2}^{\eta_{*}/2} z\theta(z)\gamma_{1}(z)dz \le \int_{-\eta_{*}/2}^{\eta_{*}/2} |z|\gamma_{1}(z)dz$.
Taking antiderivatives shows
$\int_{\eta_{*}/2<|z|<\eta_{*}} |z|\gamma_{1}(z)dz \ge \int_{-\eta_{*}/2}^{\eta_{*}/2} |z|\gamma_{1}(z)dz$.
Therefore the last two terms satisfy
\[
- \int_{-\eta_{*}/2}^{\eta_{*}/2} z\theta(z)\gamma_{1}(z)dz + \int_{\eta_{*}/2<|z|<\eta_{*}} |z|\gamma_{1}(z)dz\ge\ 0,
\]
and we conclude $\int_{\eta_{*}/2<|z|<\eta_{*}} z(\mathrm{sign}(z)+\theta(z))\gamma_{1}(z)dz \ge \Delta'$.

If $\Delta'<0$, then $s=-1$ and $\Delta=-\Delta'$.
The same argument (using $\theta\ge -1$) gives
\begin{flalign*}
&\int_{\eta_{*}/2<|z|<\eta_{*}} z(\mathrm{sign}(z)-\theta(z))\gamma_{1}(z)dz
=
\int_{\eta_{*}/2<|z|<\eta_{*}} |z|\gamma_{1}(z)dz - \int_{\eta_{*}/2<|z|<\eta_{*}} z\theta(z)\gamma_{1}(z)dz\\
&\qquad\qquad\stackrel{\eqref{del0def}}{=}-\Delta'
+\int_{\eta_{*}/2<|z|<\eta_{*}} |z|\gamma_{1}(z)dz
+\int_{-\eta_{*}/2}^{\eta_{*}/2} z\theta(z)\gamma_{1}(z)dz \\
&\qquad\qquad\geq-\Delta'
+\int_{\eta_{*}/2<|z|<\eta_{*}} |z|\gamma_{1}(z)dz
-\int_{-\eta_{*}/2}^{\eta_{*}/2} |z|\gamma_{1}(z)dz 
\ge -\Delta'\stackrel{\eqref{deldef}}{=}\Delta.
\end{flalign*}
This proves \eqref{eq:capacity}, thereby completing the proof.
\end{proof}
\begin{lemma}\label{tailv2}
Let $\alpha\stackrel{\eqref{coneq}}{\colonequals}\int_{\R}z\theta(z)\gamma_{1}(z)dz$.  If $\Delta\leq\eta_{*} d/4$, $|\alpha-\alpha_*|<1/100$, and $d-2.6|\alpha-\alpha_*|>0$ then
$$F_{\alpha,\lambda}-V_{\alpha,\lambda}(\theta)\geq\frac{(d - 2.6|\alpha-\alpha_*|)^{2}}{32.7}.$$
\end{lemma}
\begin{proof}
Define for $|z|\ge \eta$ the tail deficit $\delta(z)\colonequals1-\theta(z)\mathrm{sign}(z)\in[0,2]$, and set
\begin{equation}\label{mjdefs}
m=m_{\eta}\colonequals\int_{|z|>\eta} \delta(z)\gamma_{1}(z)\,dz,
\qquad
J=J_{\eta}\colonequals\int_{|z|>\eta} (|z|-\eta)\,\delta(z)\gamma_{1}(z)\,dz.
\end{equation}

Let 
\begin{equation}\label{tdef}
t\colonequals\frac{m}{8\gamma_{1}(\eta)}.
\end{equation}
Since $\gamma_{1}$ is decreasing on $[\eta,\infty)$,
\[
\int_{\eta<|z|<\eta+t}\gamma_{1}(z)\,dz \le 2t\gamma_{1}(\eta)
\stackrel{\eqref{tdef}}{=}\frac{m}{4}.
\]
Using $\delta\le 2$, we get
\[
\int_{\eta<|z|<\eta+t}\delta(z)\gamma_{1}(z)\,dz \le 2\int_{\eta<|z|<\eta+t}\gamma_{1}(z)\,dz \le \frac{m}{2},
\]
hence $\int_{|z|>\eta+t}\delta(z)\gamma_{1}(z)dz \stackrel{\eqref{mjdefs}}{\ge} m/2$. For $|z|\in[\eta+t,\infty)$ we have $|z|-\eta\ge t$, so
\begin{equation}\label{eq:J-m2}
J\stackrel{\eqref{mjdefs}}{=}\int_{|z|>\eta} (|z|-\eta)\delta(z)\gamma_{1}(z)dz
\ge t\int_{|z|>\eta+t} \delta(z)\gamma_{1}(z)dz
\ge t\cdot\frac{m}{2}
\stackrel{\eqref{tdef}}{=}\frac{m^2}{16\gamma_{1}(\eta)}.
\end{equation}
We then apply Lemma \ref{lower1}, which says $m_{\eta_*}\geq d/2$.  Therefore, writing $m=m-m_{\eta_*}+m_{\eta_*}$, from \eqref{mjdefs} and \eqref{etastin} below, $|m-m_{\eta_*}|\leq 4|\eta-\eta_*|/\sqrt{2\pi}\leq1.3|\alpha-\alpha_*|$, so if $d/2-1.3|\alpha-\alpha_*|>0$, then
\begin{equation}\label{fveq}
F_{\alpha,\lambda}-V_{\alpha,\lambda}(\theta)
\stackrel{\eqref{eq:gap-formula}\wedge\eqref{reedseq}}{=}\alpha J
\stackrel{\eqref{eq:J-m2}}{\geq}\alpha\frac{(d/2 - 1.3|\alpha-\alpha_*|)^{2}}{16\gamma_{1}(\eta)}
\stackrel{\eqref{alphaineq}}{\geq}
\frac{(d- 2.6|\alpha-\alpha_*|)^{2}}{32.7}.
\end{equation}
In the last inequality, we used $\eta \stackrel{\eqref{reedseq}}{=} \lambda/\alpha$ and $\eta_{*}=\lambda/\alpha^{*}$, $\alpha_*\stackrel{\eqref{reedseq}}{=}2\gamma_{1}(\eta_*)$.  If $\alpha,\alpha_{*}>1/2$ and $\lambda<.2$, we have
\begin{equation}\label{etastin}
|\eta-\eta_{*}|=\lambda|\alpha^{-1}-\alpha_{*}^{-1}|=\lambda|\alpha-\alpha_{*}|/(\alpha\alpha_{*})
\leq (4/5)|\alpha-\alpha_{*}|.
\end{equation}
\begin{equation}\label{gamineq}
|\gamma_{1}(\eta)-\gamma_{1}(\eta_{*})|\leq\frac{1}{4}|\eta-\eta_{*}|
\leq\frac{1}{5}|\alpha-\alpha_*|.
\end{equation}
$$
\Big|\frac{\alpha}{2\gamma_{1}(\eta)}
-\frac{\alpha_*}{2\gamma_{1}(\eta_*)}\Big|
=\Big|\frac{\alpha-\alpha_*}{2\gamma_{1}(\eta)}
-\alpha_* \Big(\frac{1}{2\gamma_{1}(\eta_*)}-\frac{1}{2\gamma_{1}(\eta)}\Big)\Big|
\leq\Big|\frac{\alpha-\alpha_*}{2\gamma_{1}(\eta)}\Big|
+\Big|\alpha_* \frac{\gamma_{1}(\eta)-\gamma_{1}(\eta_*)}{2\gamma_{1}(\eta)\gamma_{1}(\eta_*)}\Big|.
$$
So if $|\alpha-\alpha_*|<1/100$ with $\alpha_*=.7722\ldots$, and $\eta_*=.25573\ldots$, we get $\gamma_{1}(\eta)\geq.387$, and
\begin{equation}\label{alphaineq}
\frac{\alpha}{2\gamma_{1}(\eta)}
\geq 1 -\frac{.01}{2(.387)}-.7722\frac{.01 /5}{2\cdot.387\cdot.386}\geq.98.
\end{equation}

\end{proof}



%

\begin{lemma}[Complementary regime, Odd Case: large $\Delta$ forces a direct gap]\label{lem:Delta-large-gap}

Let $\theta\colon\R\to[-1,1]$ be odd.  Let $\alpha\stackrel{\eqref{coneq}}{\colonequals}\int_{\R}z\theta(z)\gamma_{1}(z)dz$.  
Let $d\colonequals \inf_{\overline{\theta}\in\Theta}\|\theta-\overline{\theta}\|_{1}$ and let
\[
\Delta=\Delta_{\eta_*}\colonequals\left|\int_{-\eta_*}^{\eta_*} z\,\theta(z)\gamma_{1}(z)\,dz\right|.
\]
Assume $|\alpha-\alpha_*|<1/100$.  Assume $\int_{-\eta}^{\eta}z\theta(z)\gamma_{1}(z)dz\geq0$, and
\begin{equation}\label{eq:Delta-large-assumption}
\Delta \ \ge\ \frac{\eta_*}{4}\,d.
\end{equation}
Then, if $\frac{1}{8}\,d(1 - 4|\alpha-\alpha_*|) - 6.4|\alpha-\alpha^{*}|>0$,
\begin{equation}\label{eq:Delta-large-conclusion}
F_{\alpha,\lambda}-V_{\alpha,\lambda}(\theta)
\ \ge\
\min\Big[d\Big(\frac{\lambda}{8}-|\alpha-\alpha_*|\Big),\ \frac{.98}{8}\left(\frac{1}{8}\,d(1 - 4|\alpha-\alpha_*|) - 6.4|\alpha-\alpha^{*}|\right)^2\Big].
\end{equation}
\end{lemma}

\begin{proof}

Let $\delta(z)=1-\theta(z)$ for $z\ge\eta$ and define $m,J$ as in \eqref{mjdefs}.
We split into two subcases.

\smallskip
\noindent\emph{Case 1: $J\ge \Delta/2$.}
Then by the gap formula \eqref{eq:gap-formula} and \eqref{reedseq}
\[
F_{\alpha,\lambda}-V_{\alpha,\lambda}(\theta)=\alpha J \ge \alpha \Delta/2.
\]
Using $\Delta\ge(\eta_* /4)d$ and $\alpha_*\eta_*\stackrel{\eqref{reedseq}}{=}\lambda$ gives
\[
F_{\alpha,\lambda}-V_{\alpha,\lambda}(\theta)
\ge \alpha\cdot \frac{\eta_*}{8}d 
=(\alpha-\alpha_*)\frac{\eta_*}{8}d 
+\alpha_*\frac{\eta_*}{8}d 
\geq d\Big(\frac{\lambda}{8}-|\alpha-\alpha_*|\Big).
\]
%
%



\smallskip
\noindent\emph{Case 2: $J< \Delta/2$.}
Let $s\colonequals\mathrm{sign}(\int_{0}^{\eta}z\theta(z)\gamma_{1}(z)dz)$ (and $s=1$ if the integral is zero).
We first show that (defining $\Delta_{\eta}\colonequals|\int_{-\eta}^{\eta}z\theta(z)\gamma_{1}(z)dz|$)
\begin{equation}\label{delid}
\eta m+J=s\Delta_\eta - \Big(\alpha-2\gamma_{1}(\eta)\Big).
\end{equation}

This follows since (recalling $\theta$ is odd), by definition of $\alpha$
\begin{flalign*}
\alpha
&=
2\int_0^\eta z\theta(z)\gamma_{1}(z)dz 
+ 2\int_\eta^\infty z\theta(z)\gamma_{1}(z)dz\\
&=2\int_0^\eta z\theta(z)\gamma_{1}(z)dz + 2\gamma_{1}(\eta) - 2\int_\eta^\infty z(1-\theta(z))\gamma_{1}(z)dz\\
& \stackrel{\eqref{mjdefs}}{=} s\Delta_\eta + 2\gamma_{1}(\eta) - (\eta m +J).
\end{flalign*}


By assumption $s=1$.  Now, from \eqref{delid} and $\alpha_*/2 = \gamma_{1}(\eta_*)$ (by \eqref{eq:reeds-point})
\begin{flalign*}
\eta m 
&= \Delta_\eta - J - \Big(\alpha-2\gamma_{1}(\eta)\Big)\\
& = \Delta_\eta - J - \Big(\alpha-\alpha_*-2\gamma_{1}(\eta)+2\gamma_{1}(\eta_*)\Big)\\
&\stackrel{\eqref{gamineq}}{\geq}\Delta_{\eta}-\Delta+s\Delta - J - |\alpha-\alpha_*| - 2|\alpha-\alpha_*|/5\\
&\geq\Delta/2-|\Delta_{\eta}-\Delta| - |\alpha-\alpha_*| - 2|\alpha-\alpha_*|/5\\
&\stackrel{\eqref{etastin}}{\geq}
\Delta/2 - .173|\alpha-\alpha_*|- |\alpha-\alpha_*| - 2|\alpha-\alpha_*|/5
\end{flalign*}

(By definition, $|\Delta_{\eta}-\Delta|\leq \frac{2}{\sqrt{2\pi}}\max(\eta,\eta_*)|\eta-\eta_*|\stackrel{\eqref{etastin}}{\leq}\frac{2}{\sqrt{2\pi}}(.27)(4/5)|\alpha-\alpha_*|$.)  In summary,
$$
\eta m
>\Delta/2 - 1.6|\alpha-\alpha_{*}|.
$$

Using $\Delta\geq(\eta_*/4)d$
\begin{equation}\label{mineq}
m > \frac{\Delta}{2\eta} - \frac{1.6}{\eta}|\alpha-\alpha_{*}| \ge \frac{1}{8}\frac{\eta_*}{\eta}\,d - 6.4|\alpha-\alpha^{*}|
\geq\frac{1}{8}\,d(1 - 4|\alpha-\alpha_*|) - 6.4|\alpha-\alpha^{*}|.
\end{equation}
%
%
%
Therefore, if $\frac{1}{8}\,d(1 - 4|\alpha-\alpha_*|) - 6.4|\alpha-\alpha^{*}|>0$,
\[
F_{\alpha,\lambda}-V_{\alpha,\lambda}(\theta)
\stackrel{\eqref{eq:gap-formula}}{=}\alpha J \stackrel{\eqref{eq:J-m2}}{\ge} \alpha\cdot \frac{m^2}{16\gamma_{1}(\eta)}
\stackrel{\eqref{alphaineq}}{\geq}\frac{.98}{8}\,m^2
\stackrel{\eqref{mineq}}{\ge} 
\frac{.98}{8}\left(\frac{1}{8}\,d(1 - 4|\alpha-\alpha_*|) - 6.4|\alpha-\alpha^{*}|\right)^2.
\]
Combining the Case 1 and Case 2 yields \eqref{eq:Delta-large-conclusion}.
\end{proof}

\section{Tail Equality}\label{sectail}

Let $\theta\colon\R\to[-1,1]$.  Define $\theta^{\rm odd}(z)\colonequals (\theta(z)-\theta(-z))/2$, $\forall$ $z\in\R$.
\begin{lemma}\label{elem}
Let $x,y\in[-1,1]$.  Let $s\colonequals \frac{x-y}{2}$.  Then
$$\frac{|1-x|+|-1-y|}{2}=|1 - s|.$$
Consequently, choosing $x = \theta(z)$ and $y = \theta(-z)$, upon integration we get
$$\int_{|z|>\eta}|\mathrm{sign}(z)-\theta(z)|\gamma_{1}(z)dz
=\int_{|z|>\eta}|\mathrm{sign}(z)  - \theta^{\rm odd}(z)|\gamma_{1}(z)dz.$$
\end{lemma}
\begin{proof}
Since $x\leq1$ and $y\geq-1$ the left side is
$$[1-x+1+y]/2 = [2 - x+y]/2 = 1 - (x-y)/2.$$
Similarly, since $s\leq1$, the right side is $|1-s|=1-s = 1-(x-y)/2$.  The final assertion follows from the first after writing
\begin{flalign*}
&\int_{|z|>\eta}|\mathrm{sign}(z)-\theta(z)|\gamma_{1}(z)dz
=\int_{z>\eta}\Big(|1-\theta(z)| + |-1-\theta(-z)|\Big)\gamma_{1}(z)dz\\
&\hspace{3cm}=\int_{z>\eta}2|1 - \theta^{\rm odd}(z)|\gamma_{1}(z)dz
=\int_{|z|>\eta}|\mathrm{sign}(z) - \theta^{\rm odd}(z)|\gamma_{1}(z)dz.
\end{flalign*}
\end{proof}

\section{Proof of Main Theorem}\label{seceleven}

\begin{proof}[Proof of Theorem \ref{mainthm}]

Since we let $n\to\infty$ to prove Theorem \ref{mainthm}, we assume below that $n\geq2$ (though the proof also works when $n=1$, in which case we just choose $A\colonequals\emptyset$ below).  Fix $\epsilon\colonequals10^{-7}$.  Let $f\colon\R^{n}\to[-1,1]$ maximize $\|R_{\lambda,\beta}f\|_{1}$.  (We may assume $f$ takes values in $\{-1,1\}$ by convexity.)  Let $d\colonequals\inf_{g\in\M}\|f-g\|_{1}$.  Define $\alpha\colonequals\|\int_{\R^{n}}xf(x)\gamma_{n}(x)dx\|_{\ell_{2}(\R^{n})}$.

\textbf{Case 1.}  If $d<\epsilon$, then apply \eqref{upperconstants} to get

$$
\|R_{\lambda,\beta}f\|_{1}
\leq\|R_{\lambda}\|_{\infty\to1}
-\beta(.0057),\qquad\forall\,0<\beta<10^{-10}.
$$

\textbf{Case 2.}  Suppose $d\geq\epsilon$.

Define 
\begin{equation}\label{thetadefz}
\theta(z)=\theta_{f}(z)\colonequals \E[f\,|\, P_{1}f=\alpha z],\qquad\forall\,z\in\R.
\end{equation}
Then $\alpha\stackrel{\eqref{thetadefz}}{=}\int_{\R}z\theta(z)\gamma_{1}(z)dz$.  Let $\alpha_{*}\colonequals.772216503281451\ldots$ 

\textbf{Sub-Case 2.1.}  Suppose $|\alpha-\alpha_{*}|>10^{-12}$.  Then Lemma \ref{taylorlem} implies that
\begin{equation}\label{seven1}
\|R_{\lambda}f\|_{1}\leq \|R_{\lambda}\|_{\infty\to1} - \frac{9}{10}\min((\alpha-\alpha_*)^{2},1/100).
\end{equation}
Therefore, by the $L_{1}$ triangle inequality and the last part of Lemma \ref{l1bd}
$$\|R_{\lambda,\beta}f\|_{1}
\stackrel{\eqref{rlbdef}}{\leq} \|R_{\lambda}f\|_{1}+\beta\|P_{3}f\|_{1}
\stackrel{\eqref{seven1}}{\leq}\|R_{\lambda}\|_{\infty\to1} - \frac{9}{10}\cdot10^{-24} + \beta.$$

\textbf{Sub-Case 2.2.}  Suppose $|\alpha-\alpha_{*}|<10^{-12}$.  We have by the triangle inequality and the last part of Lemma \ref{l1bd}
\begin{equation}\label{rtr}
\|R_{\lambda,\beta}f\|_{1}
\stackrel{\eqref{rlbdef}}{\leq} \|R_{\lambda}f\|_{1}+\beta\|P_{3}f\|_{1}
\leq \|R_{\lambda}f\|_{1} +\beta
\stackrel{\eqref{vdef}}{=}V_{\alpha,\lambda}(\theta_{f})+\beta.
\end{equation}
By definition \eqref{vdef}, the odd part of $\theta_{f}$ satisfies $V_{\alpha,\lambda}(\theta_{f})=V_{\alpha,\lambda}(\theta_{f}^{\rm odd})$, so
\begin{equation}\label{peneq}
\|R_{\lambda,\beta}f\|_{1}
\stackrel{\eqref{rtr}}{\leq}
V_{\alpha,\lambda}(\theta_{f}^{\rm odd})+\beta.
\end{equation}

Denote $$d'\colonequals\inf_{\overline{\theta}\in\Theta}\|\overline{\theta}-\theta_{f}^{\rm odd}\|_{1}.$$


\textbf{Sub-Case 2.2.1}.  Suppose $d'>10^{-10}$.  We can then apply \eqref{eq:Delta-large-conclusion} and Lemma \ref{tailv2} to get
%
\begin{equation}\label{beq}
\begin{aligned}
\|R_{\lambda,\beta}f\|_{1}
&\stackrel{\eqref{peneq}}{\leq} F_{\alpha,\lambda}+\beta
-\min\Big[d'\Big(\frac{\lambda}{8}-|\alpha-\alpha_*|\Big),\ \frac{.98}{8}\left(\frac{1}{8}\,d'(1 - 4|\alpha-\alpha_*|) - 6.4|\alpha-\alpha^{*}|\right)^2\Big]\\
&\leq
\|R_{\lambda}\|_{\infty\to1}+\beta
-\frac{.98}{8}\left(\frac{10^{-10}}{8}(1-10^{-10}) - 6.4\cdot 10^{-12}\right)^2\Big].
\end{aligned}
\end{equation}

(When we apply \eqref{eq:Delta-large-conclusion} we need to verify that $\Delta_{\eta}'\colonequals\int_{-\eta}^{\eta}z\theta_{f}^{\rm odd}(z)\gamma_{1}(z)dz\geq0$.  Since $\eta,m,J>0$, \eqref{delid} implies that $\Delta_{\eta}'\geq \alpha-2\gamma_{1}(\eta)$.  Since $|\alpha-\alpha_*|<10^{-12}$, we have $\alpha-2\gamma_{1}(\eta)\geq -(7/5)|\alpha-\alpha_*|\geq -(7/5)10^{-12}$ by \eqref{gamineq}.  So, if $\Delta_{\eta}'<0$, we get $|\Delta_{\eta}'|\leq (7/5)10^{-12}$, but \eqref{eq:Delta-large-conclusion} assumes $\Delta\geq\eta_* d'/4\geq10^{-10}/16$, and $$\Delta=|\Delta_{\eta_*}'-\Delta_{\eta}'+\Delta_{\eta}'|\stackrel{\eqref{etastin}}{\leq}\frac{7}{5}10^{-12}
+2|\eta-\eta_*|(.26)\gamma_{1}[-\eta_*,\eta_*]
\leq\frac{7}{5}10^{-12}+.09\cdot 10^{-12}$$ which contradicts $\Delta\geq 10^{-10}/16$.  That is, we have verified that $\Delta_{\eta}'\geq0$, as desired.)

\textbf{Sub-Case 2.2.2}.  Suppose $d'\leq 10^{-10}$.  Since each $\overline{\theta}\in\Theta$ satisfies $\overline{\theta}(z)=\mathrm{sign}(z)$ for all $|z|>\eta_*$, we have
$$\int_{|z|>\eta_*}|\mathrm{sign}(z) - \theta_{f}^{\rm odd}(z)|\gamma_{1}(z)dz\leq d'.$$
By Lemma \ref{elem}, we then have
\begin{equation}\label{taileq}
\int_{|z|>\eta_*}|\mathrm{sign}(z) - \theta_{f}(z)|\gamma_{1}(z)dz\leq d'.
\end{equation}
We will show this contradicts $d\geq\epsilon$.
Let $u\colonequals \int_{\R^{n}}xf(x)\gamma_{n}(x)dx\in\R^{n}$.  ($u\neq0$ since $\|u\|=\alpha$ and $|\alpha-\alpha_*|<10^{-12}$.)  Denote $\|u\|\colonequals\|u\|_{\ell_{2}(\R^{n})}$.  For any $v\in\R^{n}$ with $\langle v,u\rangle =0$ and $\|v\|=1$, we have by rearrangement
\begin{equation}\label{momeq}
\begin{aligned}
&\Big|\int_{|\langle\frac{u}{\|u\|}, x\rangle|>\eta_{*}}\langle v,x\rangle f(x)\gamma_{n}(x)dx
\Big|\\
&\qquad=\Big|\int_{|\langle\frac{u}{\|u\|}, x\rangle|>\eta_{*}}\langle v,x\rangle (f(x)-\mathrm{sign}(\langle u, x\rangle))\gamma_{n}(x)dx
\Big|\\
&\qquad\leq 2\int_{-\Phi^{-1}((5/4)d')}^{\infty}s\gamma_{1}(s)ds
=2\gamma_{1}(-\Phi^{-1}((5/4)d'))\\
&\qquad\leq .583\cdot 2.5\cdot d'\log(1/[(5/4)d'])
\leq 3.33\cdot 10^{-9}.
\end{aligned}
\end{equation}

%
That is, the integrand is largest when $f=-\mathrm{sign}(\langle u,x\rangle)$ for $\langle v,x\rangle>a$ for some $a>0$, so the integral is bounded by $2$ times the length of the Gaussian moment of a half space with measure at most $(5/4)d'$.
Here we used $\gamma_{1}(a)\leq .583\Phi(-a)\log(1/\Phi(-a))$ $\forall$ $a\geq2.3$.  By definition of $u$ and $\langle v,u\rangle=0$, we have
$\int_{\R^{n}}\langle v,x\rangle f(x)\gamma_{n}(x)dx=\langle v,u\rangle=0$, so that
$$
\Big|\int_{|\langle\frac{u}{\|u\|}, x\rangle|<\eta_{*}}\langle v,x\rangle f(x)\gamma_{n}(x)dx
\Big|
\stackrel{\eqref{momeq}}{\leq} 3.33\cdot 10^{-9}.
$$
Let $\delta\colonequals 3.33\cdot 10^{-9} / .0805\leq 4.2\cdot 10^{-8}$.  Let $\widetilde v\in\R^{n}$ with $\langle \widetilde v,u\rangle=0$ and $\widetilde v\neq0$ (since $n\geq2$ this is possible).  Let $A\colonequals\{x\in\R^{n}\colon|\langle\frac{u}{\|u\|}, x\rangle|<\eta_{*}, \langle\frac{\widetilde v}{\|\widetilde v\|}, x\rangle\geq0\}$.  Choose $\widetilde{v}$ such that $\int_{A}x\gamma_{n}(x)dx$ is parallel to the part of the vector $\int_{\{x\in\R^{n}\colon |\langle x,u/\|u\|\rangle|<\eta_*\}}xf(x)\gamma_{n}(x)dx$ that is perpendicular to $u$.  Then $\gamma_{n}(A)=\gamma_{1}[-\eta_*,\eta_*]/2\approx.10092$, $\int_{A}\langle x,u\rangle\gamma_{n}(x)dx=0$, $\|\int_{A}x\gamma_{n}(x)dx\|=\gamma_{1}[-\eta_*,\eta_*]\int_{0}^{\infty}s\gamma_{1}(s)ds>.0805$.
%
%
%
%
%
%
So, replacing $\widetilde{v}$ with $-\widetilde{v}$ if necessary, there exists $0<\delta'\leq\delta$ such that

\begin{equation}\label{heq1}
\Big\langle v,\,\delta'\int_{A}x\gamma_{n}(x)dx
+\int_{|\langle\frac{u}{\|u\|},x\rangle|<\eta_{*}}x f(x)\gamma_{n}(x)dx\Big\rangle
=0,\qquad\forall\,v\in\R^{n}\text{ with }\langle v, u\rangle=0.
\end{equation}
That is, the inner Gaussian moment of $f$ can be made parallel to $u$ by adding a $\delta' 1_{A}$ to $f$.  To make this moment zero, we add a multiple of $1_{B}$ where $B\colonequals\{x\in\R^{n}\colon0\leq\langle\frac{u}{\|u\|}, x\rangle<\eta_{*}\}$.  Then $\int_{B}x\gamma_{n}(x)dx=\frac{u}{\|u\|}(1- e^{-\eta_*^2 /2})/\sqrt{2\pi}\approx \frac{u}{\|u\|}.012834$, $\gamma_{n}(B)=\gamma_{1}[0,\eta_*]\approx.10092$, and by definition of $\alpha$,
\begin{flalign*}
\alpha
&=\int_{\R}\theta_{f}(z)z\gamma_{1}(z)dz
=\int_{\R}\theta_{f}^{\rm odd}(z)z\gamma_{1}(z)dz
=\int_{-\eta_*}^{\eta_*}z\theta_{f}^{\rm odd}(z)\gamma_{1}(z)dz
+\int_{|z|>\eta_*}z\theta_{f}^{\rm odd}(z)\gamma_{1}(z)dz\\
&=\int_{-\eta_*}^{\eta_*}z\theta_{f}^{\rm odd}(z)\gamma_{1}(z)dz
+\int_{|z|>\eta_*}z[\theta_{f}^{\rm odd}(z)-\mathrm{sign}(z)]\gamma_{1}(z)dz
+\int_{|z|>\eta_*}z\cdot\mathrm{sign}(z)\gamma_{1}(z)dz
\end{flalign*}
Since $\int_{|z|>\eta_*}|z|\gamma_{1}(z)dz=2\gamma_{1}(\eta_*)\stackrel{\eqref{eq:reeds-point}}{=}\alpha_*$, and $\int_{|z|>\eta_*}|z||\theta_{f}^{\rm odd}(z)-\mathrm{sign}(z)|\gamma_{1}(z)dz\leq3.33\cdot 10^{-9}$ by the same argument as \eqref{momeq}, we have by the triangle inequality
$$
\Big|\int_{-\eta_*}^{\eta_*}z\theta_{f}^{\rm odd}(z)\gamma_{1}(z)dz\Big|
\leq |\alpha-\alpha_*|+\int_{|z|>\eta_*}|z||\theta_{f}^{\rm odd}(z)-\mathrm{sign}(z)|\gamma_{1}(z)dz
\leq 10^{-12}+ 3.33\cdot 10^{-9}.
$$
Consequently, there exists $0<\omega'\leq\omega\colonequals 3.34\cdot 10^{-9}/.012834\leq 2.61\cdot 10^{-7}$ such that (after replacing $B$ if necessary with its complement in the strip $\{x\in\R^{n}\colon |\langle x,u/\|u\|\rangle|<\eta_*\}$)
$$\int_{|\langle\frac{u}{\|u\|},x\rangle|<\eta_{*}}\langle u,x\rangle(\omega'\cdot1_{B}(x) + f(x))\gamma_{n}(x)dx=0.$$

Combining this with \eqref{heq1}, we then define
\begin{equation}\label{heq2}
h\colonequals\frac{\delta'}{1+\delta'+\omega'}1_{A}
+\frac{\omega'}{1+\delta'+\omega'}1_{B}
+\frac{1}{1+\delta'+\omega'}f\cdot 1_{\big|\big\langle x, \frac{u}{\|u\|}\big\rangle\big|<\eta_{*}} 
+ \mathrm{sign}(\big\langle x, u\rangle)1_{\big|\big\langle x, \frac{u}{\|u\|}\big\rangle\big|>\eta_{*}}.
\end{equation}
Then $h$ takes values in $[-1,1]$, $\int_{\R^{n}}xh(x)\gamma_{n}(x)dx$ is parallel to $u$, and $\theta_{h}(z)\colonequals \E(h|\langle x, \frac{u}{\|u\|}\rangle=z)$ satisfies the conditions of Lemma \ref{maxchar} (i.e. $\int_{-\eta_{*}}^{\eta_{*}}\theta_{h}(z)z\gamma_{1}(z)dz=0$ by \eqref{heq1} and \eqref{heq2}, and $\theta_{h}(z)=\mathrm{sign}(z)$ $\forall$ $|z|>\eta_{*}$).  So, after rotating the domain, we have $h\in\N$ where $\N$ is defined in Lemma \ref{convlem}.  (By rotation invariance of the Gaussian measure, we may assume $u$ is parallel to the $x_1$ axis when applying this Lemma.)  Lemma \ref{convlem} says there exists $g\in\M$ with $\|f-g\|_{1}\leq\|f-h\|_{1}$.  From the Case 2 assumption and the definition of $d=\inf_{g\in\M}\|f-g\|_{1}$, this means $\|f-h\|_{1}\geq\epsilon$.
But \eqref{heq2} implies
\begin{flalign*}\|f-h\|_{1}
&=\int_{\big|\big\langle x,\frac{u}{\|u\|}\big\rangle\big|<\eta_{*}}|f(x)-h(x)|\gamma_{n}(x)dx
+\int_{\big|\big\langle x,\frac{u}{\|u\|}\big\rangle\big|>\eta_{*}}|f(x)-h(x)|\gamma_{n}(x)dx\\
&\leq (2(\delta+w)+w+2w+\delta+w)\gamma_{1}[-\eta_{*},\eta_{*}]/4 + \int_{|z|>\eta_{*}}|\theta_{f}(z)-\mathrm{sign}(z)|\gamma_{1}(z)dz\\
&\stackrel{\eqref{taileq}}{\leq} 
(6\delta+6w)(.051)+d'
\leq
(6\cdot 4.2\cdot10^{-8}+6\cdot 2.61\cdot 10^{-7})(.051)
+ 10^{-10}
<10^{-7}.
\end{flalign*}
%
This contradicts $\|f-h\|_{1}\geq\epsilon=10^{-7}$.  That is, Sub-Case 2.2.2 does not occur.

\textbf{Combine all cases}.  Combining the above cases gives (for all $0<\beta<10^{-10}$)

\begin{equation}\label{finalineq}
\begin{aligned}
\|R_{\lambda,\beta}\|_{\infty\to1}
&\leq \|R_{\lambda}\|_{\infty\to1}
+\max\Big[-\beta(.0057), \beta-9\cdot 10^{-25}, \beta - \frac{.9799}{8}\left(\frac{10^{-10}}{8} - 6.4\cdot 10^{-12}\right)^2\Big]\\
&\leq\|R_{\lambda}\|_{\infty\to1}
+\max\Big[-\beta\cdot 5.7\cdot 10^{-3}, \beta-9\cdot 10^{-25}, \beta - 4.5\cdot 10^{-24}\Big].
\end{aligned}
\end{equation}
%
%

Choosing $\beta\colonequals 8\cdot 10^{-25}$ gives
$$
\|R_{\lambda,\beta}\|_{\infty\to1}
\leq \|R_{\lambda}\|_{\infty\to1}
-45.6\cdot 10^{-28}.
$$

Finally, $\lim_{n\to\infty}\sup_{g\colon\R^{n}\to B_{n}}\int_{\R^{n}}\|R_{\lambda}g(x)\|_{\ell_{2}(\R^{n})}\gamma_{n}(x)dx=1-\lambda$ by e.g. \cite[Theorem 1]{reeds93}, since the function $f(x)\colonequals x/\|x\|_{2}$ for all $x\in\R^{n}\setminus\{0\}$ has all of its $L_{2}$ Hermite-Fourier mass on the first level, as $n\to\infty$.

So, \eqref{klb} completes the proof, since if $c\colonequals (1-\lambda) / \|R_{\lambda}\|_{\infty\to1}$, with $\lambda=\lambda_*\approx .197\ldots$, then $\|R_{\lambda}\|_{\infty\to1}\approx.4788\ldots$, $c/\|R_{\lambda}\|_{\infty\to1}\geq3.5$, and
$$K_{G}\stackrel{\eqref{klb}}{\geq}\frac{1-\lambda}{\|R_{\lambda}\|_{\infty\to1} - 45.6\cdot 10^{-28}} 
\geq c + c\frac{45.6\cdot 10^{-28}}{\|R_{\lambda}\|_{\infty\to1}}
\geq c + 159.6\cdot 10^{-28}.$$
%
%


\end{proof}

\textbf{Acknowledgement}.  Some of the above material was created with the assistance of OpenAI's ChatGPT 5.2, including some ancillary lemmas and propositions.  However, ChatGPT did not contribute to the strategy of this paper; in fact, many of its strategic suggestions were of negative value.



\bibliographystyle{amsalpha}

\begin{thebibliography}{}

\bibitem[AN04]{alon04}
Noga Alon and Assaf Naor. \emph{Approximating the cut-norm via Grothendieck's inequality}. In Proceedings of the thirty-sixth annual ACM symposium on Theory of computing (STOC '04). Association for Computing Machinery, New York, NY, USA, 72–80, 2004.

\bibitem[BMMN13]{braverman13}
Mark Braverman, Konstantin Makarychev, Yury Makarychev, and Assaf Naor. \emph{The Grothendieck constant is strictly smaller than Krivine's bound}. Forum of Mathematics, Pi. 2013;1:e4. doi:10.1017/fmp.2013.4

\bibitem[D84]{davie84}
A. M. Davie. \emph{A lower bound for $K_{G}$}. Unpublished manuscript, 1984.

\bibitem[G75]{gross75}
Leonard Gross.
\emph{Logarithmic Sobolev Inequalities}.
American Journal of Mathematics.
Vol. 97, No. 4, pp. 1061-1083, 1975.

\bibitem[G53]{groth53}
A. Grothendieck. \emph{Resume de la theorie metrique des produits tensoriels topologiques}. Bol. Soc. Mat. Sao Paulo, 8:1-79, 1953.

\bibitem[K02]{khot02}
Subhash Khot.
\emph{On the power of unique 2-prover 1-round games}. In Proceedings of the thirty-fourth annual ACM symposium on Theory of computing (STOC '02). Association for Computing Machinery, New York, NY, USA, 767–775, 2002.

\bibitem[KN12]{khot12}
S.~A. Khot and A. Naor.
\emph{Grothendieck-type inequalities in combinatorial optimization}. Comm. Pure Appl. Math. {\bf 65}, no.~7, 992--1035, 2012.

\bibitem[K01]{konig01}
H. K\"{o}nig. \emph{On an extremal problem originating in questions of unconditional convergence}. In Recent progress in multivariate approximation (Witten-Bommerholz, 2000), volume 137 of Internat. Ser. Numer. Math., pages 185-192. Birkh\"{a}user, Basel, 2001.

\bibitem[K77]{krivine77}
J.-L. Krivine. \emph{Sur la constante de Grothendieck}. C. R. Acad. Sci. Paris Ser. A-B, 284(8):A445-A446, 1977.

\bibitem[LP68]{linden77}
J. Lindenstrauss and A. Pelczy\'{n}ski. 
\emph{Absolutely summing operators in Lp-spaces and their applications}. Studia Math., 29:275-326, 1968.


\bibitem[P12]{pisier12}
Gilles Pisier.
\emph{Grothendieck's theorem, past and present}. Bull. Amer. Math. Soc. (N.S.) {\bf 49}, no.~2, 237--323, 2012.

\bibitem[RS09]{ragh09}
P. Raghavendra and D. Steurer. 
\emph{Towards computing the Grothendieck constant}. In Proceedings of the Twentieth Annual ACM-SIAM Symposium on Discrete Algorithms, pages 525-534, 2009.

\bibitem[R91]{reeds93}
J. A. Reeds. \emph{A new lower bound on the real Grothendieck constant}. Unpublished manuscript, 1991. 

\bibitem[RN14]{regev14}
Assaf Naor and Oded Regev. \emph{Krivine Schemes are Optimal}. Proceedings of the American Mathematical Society, vol. 142, no. 12, pp. 4315–20, 2014. 

\bibitem[T87]{tsirelson87}
Tsirel'son, B.S. \emph{Quantum analogues of the Bell inequalities. The case of two spatially separated domains}. J Math Sci 36, 557–570, 1987.

\end{thebibliography}

\end{document}